\documentclass[11pt, reqno]{amsart}
\usepackage{mathrsfs}
\usepackage[active]{srcltx}
\usepackage{mathrsfs,amsmath}
\usepackage{mathtools}
\usepackage{longtable}
\usepackage{todonotes}
\usepackage{amssymb}
\usepackage{cite}
\allowdisplaybreaks

\usepackage{xcolor}
\definecolor{bluecite}{HTML}{0875b7}

\usepackage[unicode=true,
bookmarksopen={true},
pdffitwindow=true,
colorlinks=true,
linkcolor=bluecite,
citecolor=bluecite,
urlcolor=bluecite,
hyperfootnotes=false,
pdfstartview={FitH},
pdfpagemode= UseNone]{hyperref}

\newcommand{\ler}[1]{\left(#1\right)}
\newcommand{\lers}[1]{\left\{#1\right\}}

\newcommand{\dd}{\mathrm{d}}

\newcommand{\cW}{\mathcal{W}}
\newcommand{\R}{\mathbb{R}}
\newcommand{\HH}{\mathbb{H}}
\newcommand{\wth}{\cW_p(\HH^n)}
\newcommand{\be}{\begin{equation}}
\newcommand{\ee}{\end{equation}}

\newcommand{\abs}[1]{\left|#1\right|}
\newcommand{\norm}[1]{\left|\left|#1\right|\right|}

\newcommand{\xt}{\widetilde{x}}
\newcommand{\yt}{\widetilde{y}}
\newcommand{\zt}{\widetilde{z}}

\newcommand{\xtyt}{(\widetilde{x},\widetilde{y})}
\newcommand{\Wphn}{\mathcal{W}_p(\HH^n)}

\newcommand{\dwp}{d_{\mathcal{W}_p}}

\newcommand{\Wpl}{\mathcal{W}_p(L)}

\newcommand{\supp}{\mathrm{supp}}

\newcommand{\isom}{\mathrm{Isom}}
\newcommand{\bigO}{\underline{\underline{O}}}

\usepackage[left=3.9cm,right=3.9cm,top=3.9cm,bottom=3.9cm]{geometry}

\newtheorem{theorem}{Theorem}[section]
\newtheorem{proposition}[theorem]{Proposition}
\newtheorem{lemma}[theorem]{Lemma}
\newtheorem{corollary}[theorem]{Corollary}

\newtheorem{remark}[theorem]{Remark}
\numberwithin{equation}{section}

\usepackage{bbm}  
\usepackage{dsfont}


\title[Isometries and isometric embeddings of $\Wphn$]{Isometries and isometric embeddings of Wasserstein spaces over the Heisenberg group}

\author[Zolt\'an M. Balogh]{Zolt\'an M. Balogh}
\address{Zolt\'an M. Balogh, Universit\"at Bern\\ Mathematisches Institut (MAI)\\ Sidlerstrasse 12\\ 3012 Bern\\ Schweiz}
\email{zoltan.balogh@unibe.ch}

\author[Tam\'as Titkos]{Tam\'as Titkos}
\address{Tam\'as Titkos, HUN-REN Alfr\'ed R\'enyi Institute of Mathematics\\ Re\'altanoda u. 13-15.\\
Budapest H-1053\\ Hungary\\ and Corvinus University of Budapest\\ Fővám tér 13-15.\\
Budapest H-1093\\ Hungary}
\email{titkos.tamas@renyi.hu}

\author[D\'aniel Virosztek]{D\'aniel Virosztek}
\address{D\'aniel Virosztek, HUN-REN Alfr\'ed R\'enyi Institute of Mathematics\\ Re\'altanoda u. 13-15.\\
Budapest H-1053\\ Hungary}
\email{virosztek.daniel@renyi.hu}

\thanks{Z. M. Balogh is supported by the Swiss National Science Foundation, Grant Nr. {200020\_191978}.  }
\thanks{T. Titkos is supported by the Hungarian National Research, Development and Innovation Office - NKFIH grant no. K134944) and by the Momentum program of the Hungarian Academy of Sciences under grant agreement no. LP2021-15/2021.}
\thanks{D. Virosztek is supported by the Momentum program of the Hungarian Academy of Sciences under grant agreement no. LP2021-15/2021, by the Hungarian Research, Development and Innovation Office (NKFIH) under grant agreement no. Excellence\_151232, and partially supported by the ERC Synergy Grant No. 810115.}

\begin{document}

\begin{abstract}
Our purpose in this paper is to study isometries and isometric embeddings of the $p$-Wasser\-stein space $\Wphn$ over the Heisenberg group $\HH^n$ for all $p>1$ and for all $n\geq1$. First, we create a link between optimal transport maps in the Euclidean space $\mathbb{R}^{2n}$ and the Heisenberg group $\HH^n$. Then we use this link to understand isometric embeddings of $\mathbb{R}$ and $\mathbb{R}_+$ into $\Wphn$ for $p>1$. That is, we characterize complete geodesics and geodesic rays in the Wasserstein space. Using these results we determine the metric rank of $\Wphn$. Namely, we show that $\mathbb{R}^k$ can be embedded isometrically into $\Wphn$ for $p>1$ if and only if $k\leq n$. As a consequence, we conclude that $\mathcal{W}_p(\mathbb{R}^k)$ and $\mathcal{W}_p(\HH^k)$ can be embedded isometrically into $\Wphn$ if and only if $k\leq n$. In the second part of the paper, we study the isometry group of $\Wphn$ for $p>1$. We find that these spaces are all isometrically rigid, meaning that for every isometry $\Phi:\Wphn\to\Wphn$ there exists an isometry $\psi:\HH^n\to\HH^n$ such that $\Phi=\psi_{\#}$.
\end{abstract}

\maketitle

\tableofcontents

\section{Introduction: motivation and main results}\label{s1:intro}

 In recent decades, there has been rapid development in the theory of optimal mass transportation and its countless applications. 
 The original transport problem initiated by Monge is to find the cheapest way to transform one probability distribution into another when the cost of transporting mass is proportional to the distance. Probably the most important metric structure which is related to optimal mass transportation is the so-called $p$-Wasserstein space $\mathcal{W}_p(X)$, where the underlying space $X$ is a complete and separable metric space (see the precise definition later). 
 
 Various connections between the geometry of the underlying space $X$ and the geometry of the Wasserstein space $\mathcal{W}_2(X)$ have been investigated by Lott and Villani in the groundbreaking paper \cite{LV}. The pioneering work of Kloeckner \cite{K} and the follow-up papers \cite{BK,BK2, K2} started to explore fundamental geometric features of $2$-Wasserstein spaces, including the description of complete geodesics and geodesic rays, determining their different type of ranks, and understanding the structure of their isometry group $\isom(\mathcal{W}_2(X))$. It is a general phenomenon that the group of isometries contains elements that are closely related to certain morphisms of the underlying structure, see e.g. \cite{DKM,dolinar-molnar,kuiper,isemb-jmaa,geher-titkos,KT,molnar-levy}. In the case of $p$-Wasserstein spaces, the isometry group $\isom(X)$ of the underlying space $X$ is always isomorphic to a subgroup of $\isom\big(\mathcal{W}_p(X)\big)$. In fact, in the typical case, $\isom(X)$ and $\isom\big(\mathcal{W}_p(X)\big)$ are isomorphic. In such a case we call the Wasserstein space isometrically rigid.
 
Kloeckner showed in \cite{K} that $2$-Wasserstein spaces over Euclidean spaces have the property that their isometry group is strictly larger than the isometry group of the underlying Euclidean space. {In their recent manuscript, Che, Galaz-Garc\'ia, Kerin, and Santos-Rodr\'iguez extended Kloeckner's result by showing that the isometry group of a $2$-Wasser\-stein space contains non-trivial isometries if the underlying space is of the form $X=H\oplus_{\ell_2} Y$, where $H$ is a Hilbert space and $Y$ is a proper metric space.}

As it was proven in \cite{GTV1} and \cite{GTV2}, the parameter $p=2$ is indeed special, as for all $p\neq 2$ the $p$-Wasserstein space $\mathcal{W}_p(\mathbb{R}^n)$ is isometrically rigid. Interestingly enough, in the case of $X=[0,1]$ the situation is very much different: $\mathcal{W}_p([0,1])$ is isometrically rigid if and only if $p\neq1$, see \cite{GTV1}. In recent years, rigidity results concerning $p$-Wasserstein spaces were proven in various non-euclidean setups as well. The case of the $n$-dimensional tori and spheres is settled in \cite{TnSn}. Bertrand and Kloeckner \cite{BK, BK2} proved the isometric rigidity of the Wasserstein space $\mathcal{W}_2(X)$ over Hadamard manifolds $X$. Furthermore, in \cite{S-R} Santos-Rodr\'iguez considered a broad class of manifolds. He showed that $\mathcal{W}_2(X)$ is isometrically rigid whenever $X$ is a closed Riemannian manifold with strictly positive sectional curvature. Furthermore, for compact rank one symmetric spaces (CROSSes), he proved isometric rigidity for all $p>1$.

This paper aims to study isometries and isometric embeddings of the $p$-Wasser\-stein space $\Wphn$ over the Heisenberg group $\HH^n$ endowed with the Heisenberg-Kor\'{a}nyi metric for all $p>1$ and for all $n\geq1$. The metric structure of the Heisenberg group is rather different from Euclidean spaces or Riemannian manifolds from the viewpoint of rectifiability or Lipschitz extensions \cite{AK, BF, BHW}. Starting from the important contribution of Ambrosio and Rigot \cite{AR} and followed by the papers of Juillet \cite{J, J1}, Figalli and Juillet \cite{FJ} in recent years considerable research has been devoted in order to develop the theory of mass transportation in this geometric setting. In fact, it turns out that the theory of optimal mass transport leads to a deeper understanding of the metric structure of the Heisenberg group and the related geometric inequalities as shown by the results of \cite{J} and  \cite{BKS, BKS1}. These results serve as a strong motivation for further investigating the interplay between the geometry of the underlying space $\HH^n$ and the corresponding $p$-Wasser\-stein space $\Wphn$. 

Before stating our first main result, let us recall that several interesting embedding and non-embedding results were proved earlier by Bertrand
and Kloeckner in \cite{BK2,K,K2} for Wasserstein spaces. In \cite{K2} Kloeckner showed for any metric space $(X,\varrho)$ and any parameter $p\geq1$ that the power $X^k$ admits a bi-Lipschitz embedding into $\mathcal{W}_p(X)$ for all $k\in\mathbb{N}$. So in particular, any power of the Heisenberg group and thus any power of the real line can be embedded such a way into $\Wphn$. Concerning isometric embeddings, Kloeckner showed in \cite{K} that
if $X$ contains a complete geodesic then $\mathcal{W}_2(X)$ contains an isometric copy of the open Euclidean
cone $\mathbb{R}^k_+$ of arbitrary dimension. In particular, it contains isometric embeddings of
Euclidean balls of arbitrary dimension and radius, and bi-Lipschitz embeddings of
$\mathbb{R}^k$ for all $k\in\mathbb{N}$. However, it turned out that if the whole $\mathbb{R}^k$ embeds isometrically into $\mathcal{W}_2(\mathbb{R}^n)$ then $k\leq n$. Our first result says that the same holds true for $\Wphn$ for all $p>1$.

\begin{theorem}\label{embprop}
Let $n \in\mathbb{N}$ and $p>1$. The rank of $\Wphn$ is $n$, that is, $\mathbb{R}^k$ can be embedded isometrically into $\Wphn$ if and only if $k\leq n$.
\end{theorem}

The strategy of the proof is the following. 
First, in Lemma \ref{lift} we create a link between optimal transport maps in the Euclidean space $\mathbb{R}^{2n}$ and the Heisenberg group $\HH^n$. Then we use this link to understand complete geodesics, i.e., isometric embeddings of $\mathbb{R}$  into $\Wphn$. Namely, we show that complete geodesics are induced by right-translations of the multiples of the same horizontal vector. Using these facts we can give the proof of Theorem \ref{embprop}. As a corollary, we obtain that either of the spaces $\mathcal{W}_p(\mathbb{R}^k)$, or $\mathcal{W}_p(\HH^k)$ can be embedded isometrically into $\Wphn$ if and only if $k\leq n$.\\

In the second part of the paper, we study the isometry group of $\Wphn$ for $p>1$. The main result is the following:

\begin{theorem}\label{isom-rigidity}
Let $p>1$ and $n\geq1$ be fixed. Then $\Wphn$ is isometrically rigid, i.e., for any isometry $\Phi:\Wphn\to\Wphn$ there exists an isometry $\psi:\HH^n\to\HH^n$ such that $\Phi=\psi_{\#}.$
\end{theorem}
The proof of this theorem is based on the description of vertically supported measures as endpoints of geodesic rays (isometric embeddings of $\mathbb{R_+}$ into $\Wphn$). Using this description we can prove that up to an isometry of the base space, all the Dirac masses are fixed by an isometry of $\Wphn$. Moreover, this is true for all vertically supported measures as well.  The technique of the vertical Radon transform is used to finish the proof.

We mention finally that the method of this paper does not work to prove the rigidity of the first Wasserstein space $\mathcal{W}_1({\mathbb H}^n)$. The reason for this is that in this case, we cannot give a metric characterization of complete geodesics and geodesic rays as in the case $\Wphn$ for $p>1$, (see Remark \ref{p=1 geod} for details).  The isometric rigidity of $\mathcal{W}_1({\mathbb H}^n)$ will be treated by a different method in our forthcoming paper \cite{BTV}.

 \section{Preliminary notions, notations, and terminology}
 
 We start with notations that will be used in the sequel, for more details and references on Wasserstein spaces we refer the reader to any of the following comprehensive textbooks \cite{AG,Figalli, Santambrogio,Villani,V}. Let us recall first what a $p$-Wasserstein space is. Let $p\geq1$ be a fixed real number, and let $(X,\varrho)$ be a complete and separable metric space. We denote by $\mathcal{P}(X)$ the set of all Borel probability measures on $X$. The symbol $\supp\mu$ stands for the support of $\mu\in\mathcal{P}(X)$. The set of Dirac measures will be denoted by $\Delta_1(X)=\{\delta_x\,|\,x\in X\}$. A probability measure $\Pi$ on $X \times X$ is called a coupling for $\mu,\nu\in\mathcal{P}(X)$ if the marginals of $\Pi$ are $\mu$ and $\nu$, that is, $$\Pi\ler{A \times X}=\mu(A)\qquad\mbox{and}\qquad\Pi\ler{X \times B}=\nu(B)$$ for all Borel sets $A,B\subseteq X$. The set of all couplings is denoted by $C(\mu,\nu)$. Using couplings, we can define the $p$-Wasserstein distance and the corresponding $p$-Wasserstein space as follows: the $p$-Wasserstein space $\mathcal{W}_p(X)$  is the set of all $\mu\in\mathcal{P}(X)$ that satisfy
\begin{equation}
\int_X \varrho^p(x,\hat{x})~\mathrm{d}\mu(x)<\infty
\end{equation}
for some (and hence all) $\hat{x}\in X$, endowed with the $p$-Wasserstein distance
\begin{equation}
 \label{eq:wasser_def}
\dwp\ler{\mu, \nu}:=\ler{\inf_{\Pi \in C(\mu, \nu)} \iint_{X \times X} \varrho^p(x,y)~\dd \Pi(x,y)}^{1/p}.
\end{equation}
It is known (see e.g. Theorem 1.5 in \cite{AG} with $c=\varrho^p$) that the infimum in \eqref{eq:wasser_def} is, in fact, a minimum in this setting. Those couplings that minimize \eqref{eq:wasser_def} are called optimal transport plans.

As the terminology suggests, all notions introduced above are strongly related to the theory of optimal transportation. Indeed, for given sets $A$ and $B$ the quantity $\Pi(A,B)$ is the weight of mass that is transported from $A$ to $B$ as $\mu$ is transported to $\nu$ along the transport plan $\Pi$.

Given two metric spaces $(Y,d_Y)$ and $(Z,d_Z)$, a map $f:Y\to Z$ is an isometric embedding if $d_Z(f(y),f(y'))=d_Y(y,y')$ for all $y,y'\in Y$. A self-map $\psi\colon Y\to Y$ is called an isometry if it is a surjective isometric embedding of $Y$ onto itself. The symbol $\isom(\cdot)$ will stand for the group of all isometries. For an isometry $\psi\in\isom(X)$ the induced push-forward map is
$$\psi_\# \colon \mathcal{P}(X)\to\mathcal{P}(X);\qquad\big(\psi_\#(\mu)\big)(A)=\mu(\psi^{-1}[A])$$ 
for all Borel sets $A\subseteq X$ and $\mu\in\mathcal{P}(X)$, where 
$\psi^{-1}[A]=\{x\in X\,|\, \psi(x)\in A\}.$ We call $\psi_\#(\mu)$ the \emph{push-forward} of $\mu$ with $\psi$.

A very important feature of $p$-Wasserstein spaces is that $\mathcal{W}_p(X)$ contains an isometric copy of $X$. Indeed, since $C(\delta_x,\delta_y)$ has only one element (the Dirac measure $\delta_{(x,y)})$ {for all $x,y\in X$}, we have that \begin{equation*}\dwp(\delta_x,\delta_y)=\left(\iint_{X\times X}\varrho^p(u,v)~\mathrm{d}\delta_{(x,y)}(u,v)\right)^{1/p}=\varrho(x,y),
\end{equation*}
and thus the embedding
\begin{equation}
\iota\colon X\to\mathcal{W}_p(X),\qquad \iota(x):=\delta_x
\end{equation}
is distance preserving. Furthermore, the set of finitely supported probability measures (in other words, the collection of all finite convex combinations of Dirac measures)
\begin{equation}
    \mathcal{F}(X)=\Bigg\{\sum_{j=1}^k\lambda_j\delta_{x_j}\,\Bigg|\,k\in\mathbb{N},x_j\in X,\,\lambda_j\geq0\,(1\leq j\leq k),\,\sum_{j=1}^k\lambda_j=1\Bigg\}
\end{equation}
is dense in $\mathcal{W}_p(X)$, see e.g. Example 6.3 and Theorem 6.18 in \cite{Villani}. Another important feature is that isometries of $X$ appear in $\isom(\mathcal{W}_p(X))$ by means of a natural group homomorphism
\begin{equation}
 \label{eq:hashtag}
\#\colon \, \mathrm{Isom} (X) \rightarrow \mathrm{Isom}  \ler{\mathcal{W}_p(X)}, \qquad \psi \mapsto \psi_{\#}.
\end{equation}
Isometries that belong to the image of $\#$ are called trivial isometries. If $\#$ surjective, i.e., if every isometry is trivial, then we say that $\mathcal{W}_p(X)$ is isometrically rigid.

In our considerations, the base space $(X,\varrho)$ will be the Heisenberg group endowed with the Heisenberg-Kor\'anyi metric.\\

Let us recall that the underlying space of the Heisenberg group $\mathbb{H}^n$ is $\mathbb{H}^n = \mathbb{R}^n \times \mathbb{R}^n \times \mathbb R$ with the group operation given by 
$$(x,y,z)\ast (x',y', z')= (x+x',y+y',z+z'+2\sum_{i=1}^n (x'_iy_i-x_iy'_i)).$$

The Heisenberg group has a rich group of transformations including left-translations $\tau_{(\widetilde{x},\widetilde{y},
\widetilde{z})}: \mathbb{H}^n \to \mathbb{H}^n$ given by $$\tau_{(\widetilde{x},\widetilde{y},
\widetilde{z})}(x,y,z)= (\widetilde{x},\widetilde{y},
\widetilde{z})\ast (x,y,z)$$ and non-isotropic dilations $\delta_r:\mathbb{H}^n
\to \mathbb{H}^n$, $r>0$ $$\delta_r(x,y,z)= (r x, r y,r^2 z).$$
The left-invariant Heisenberg-Kor\'anyi metric $d_H$ is defined using the group structure of $\mathbb{H}^n$ by the formula $$d_H((x,y,z),
(x',y',z')) = \|(-x,-y,-z) \ast (x',y',z')\|_H, $$ where
$\|\cdot\|_H$ stands for the homogeneous norm on $\mathbb{H}^n$  $$\|(x,y,z)\|_H =
\Big(\big(\sum_{i=1}^n(x^2_i+y^2_i)\big)^2 + z^2\Big)^{1/4}.$$ 

For more information about various left-invariant metrics, isometries and isometric embeddings of the Heisenberg group we refer the interested reader to \cite{BFS}. Let us recall that in the Heisenberg group the $0z$ axis (that is, the set $\{(0,0,z)\,|\,z\in\mathbb{R}\}$) plays a special role, being the center of Heisenberg group.  Left translations of the $0z$ axis are called vertical lines. For a given $(x,y)\in\mathbb{R}^{2n}$ the symbol $L_{(x,y)}$ stands for the vertical line passing through the horizontal vector $(x,y,0)$
$$L_{(x,y)}=\{(x,y,r)\,|\,r\in\mathbb{R}\},$$
and $\mathcal{L}$ denotes the set of all vertical lines
$$\mathcal{L}=\{L_{(x,y)}\,|\,(x,y)\in\mathbb{R}^{2n}\}.$$ 
We call a measure $\mu\in\Wphn$ vertically supported if its support is contained in a vertical line.  For a vertical line $L\in\mathcal{L}$ let us introduce the notation
$$\mathcal{W}_p(L):=\Big\{\mu\in\Wphn\,\Big|\,\supp\mu\subseteq L\Big\}.$$

It is known that any isometry of the Heisenberg group maps vertical lines onto vertical lines (see for example Theorem 1.1 and Lemma 2.3 in \cite{BFS}).
Similarly, complete geodesics and geodesic rays will be preserved by isometries in the sense that the image of a complete geodesic under an isometry is a complete geodesic, and the image of a geodesic ray is again a geodesic ray.
Recall that a complete geodesic is a curve $\gamma: \mathbb{R} \to \mathcal{W}_p(\mathbb H^n)$ such that 
$$d_{\mathcal{W}_p}(\gamma(t),\gamma(s)) = C|t-s|$$ for all $t,s \in \mathbb{R}$ and a constant $C>0$. A geodesic ray is a curve $\gamma: [a, \infty) \to \mathcal{W}_p(\mathbb H^n)$ with the same property. 
Note, that by reparametrising the curve $\gamma$ we can always achieve that $C=1$, and thus $\gamma$ will be an isometric embedding of the real line (or half-line) into  $\mathcal{W}_p(\mathbb H^n)$. Geodesics with $C=1$ will be called unit-speed geodesics.


\section{Complete geodesics and geodesic rays in $\mathcal{W}_p(\mathbb H^n)$.}

\label{s:embeddings}
Our first aim is to understand the structure of isometric embeddings of $\mathbb{R}$ and $\mathbb{R}_+$ into $\Wphn$. On the one hand, isometric copies of the real line will help us later in this section to determine the rank of $\Wphn$. On the other hand, isometric copies of the nonnegative half-line in $\Wphn$ will come in handy for characterizing vertically supported measures. Such measures will play a crucial role in the next section where we will investigate isometric rigidity of Wasserstein spaces.

Optimal transport maps between absolutely continuous measures in the Heisenberg group were studied by Ambrosio and Rigot in \cite{AR}. Since in this paper we shall work with more general (mainly finitely supported) measures; we will need a different way of understanding optimal transport maps acting between them. The approach that we use here is based on the notion of cyclical monotonicity.

Let us recall that a subset $\Gamma \subseteq \HH^n\times \HH^n$ is called to be {$d_H^p$-}cyclically monotone if for any finite selection of points $\{(q_i, q_i')\}_{i=1}^N \subset \Gamma $ we have 
\be \label{cyclic-mon}
\sum_{i=1}^N d_H^p(q_i, q_i') \leq \sum_{i=1}^N d_H^p(q_{i+1}, q_i').
\ee
Here and also in the sequel, we will use the convention that  $q_{N+1} = q_1$. The following is a consequence of Theorem 3.2 of the paper of Ambrosio and Pratelli \cite{AP}. 

\begin{theorem} \label{Ambrosio-Pratelli}
Let $\mu, \nu \in \wth$ and $\Pi$ be a coupling between $\mu$ and $\nu$. Then $\Pi$ is optimal if and only if it is supported on a {$d_H^p$-}cyclically monotone set $\Gamma \subseteq \HH^n\times \HH^n$.
    \end{theorem}
For given $\mu,  \nu \in \mathcal{W}_p(\mathbb H^n)$ there are situations when the optimal coupling $\Pi$ can be achieved by a transport map $\widehat{T}: \mathbb{H}^n \to \mathbb{H}^n$ with  $(\widehat{T})_{\#} \mu = \nu$ such that $\Pi = (\mathrm{Id}\times \widehat{T})_{\#} \mu$, where $$\mathrm{Id}\times \widehat{T}: \mathbb{H}^n \to \HH^n\times\HH^n, \ (\mathrm{Id}\times \widehat{T})(q)= (q, \widehat{T}(q)), \ \text{for} \ q \in \mathbb{H}^n.$$

In this case, the Wasserstein distance $d_{\mathcal{W}_p} (\mu, \nu)$ can be computed by the formula
$$d^p_{\mathcal{W}_p} (\mu, \nu)= \int_{\mathbb{H}}d^p_H(q, \widehat{T}(q)) ~\mathrm{d}\mu(q).$$ 

We shall create optimal transport maps on $\HH^n$ that are derived from optimal transport maps in the Euclidean space $\mathbb{R}^{2n}$. In order to formulate this result we use the notation $\pi: \HH^n \to \mathbb{R}^{2n}$ for the standard projection $\pi(x,y,z) = (x,y)$. We refer to vectors of the form $(x,y,0)\in\HH^n$ as horizontal vectors, and sometimes we identify them with $\pi(x,y,0)=(x,y)\in\mathbb{R}^{2n}$. For these vectors we have $\|(x,y,0)\|_H=\|(x,y)\|$, where $\|\cdot\|$ denotes the Euclidean norm. By $x\cdot y$ we denote the standard scalar product of two vectors $x,y \in \mathbb{R}^n$, and $d_E$ denotes the Euclidean distance.
\begin{lemma} \label{lift}
Let $\nu \in \wth$ and define $\mu= \pi_{\#} \nu$. Then $\mu \in \cW_{p}(\mathbb{R}^{2n})$. Assume, that the mapping $T: \mathbb{R}^{2n} \to \mathbb{R}^{2n}$, given by  $ T(x,y)= (T_1(x,y), T_2(x,y))$  for all $(x,y) \in \mathbb{R}^{2n}$ is an optimal transport map between $\mu$ and $T_{\#} \mu$.  Then the lifted mapping $\widehat{T}: \HH^n \to \HH^n$ defined by:
\be \label{lifted-map}
\widehat{T}(x,y,z) = (T_1(x,y), T_2(x,y), z + 2( y\cdot T_1(x,y) - x\cdot T_2(x,y)), 
\ee
for  $(x,y,z)\in \HH^n$
 is an optimal transport map between $\nu$ and $\widehat{T}_{\#} \nu$.
\end{lemma}
\begin{proof} Let us take $\nu \in \wth$. Then it is easy to check that $\mu = \pi_{\#}\nu$ is in $\cW_{p}(\mathbb{R}^{2n})$. Let $T: \mathbb{R}^{2n} \to \mathbb{R}^{2n}$ be an optimal transport map between $\mu$ and $T_{\#}\mu$ and $\widehat{T}$ be its lift defined as in \eqref{lifted-map}. We intend to show that the coupling $\widehat{\Pi} := (\widehat{\mathrm{Id}}\times \widehat{T})_{\#}\nu$ is optimal. Here $\widehat{\mathrm{Id}}: \HH^n \to \HH^n$ denotes the identity map of $\HH^n$. On account of Theorem \ref{Ambrosio-Pratelli} we need to check that $\widehat{\Pi}$ is supported on a cyclically monotone set. 

To do that, let $(q_i,\widehat{T}(q_i)))_{i=1}^N$ be a finite set of points in the support of $\widehat{\Pi}$. It is easy to see that the set of points $(\pi(q_i), T(\pi(q_i)))_{i=1}^N$ are in the support of the coupling $\Pi=(\mathrm{Id}\times T)_{\#}\mu$, where $\mathrm{Id}$ is the identity map on $\mathbb{R}^{2n}$. Since by assumption $T$ is an optimal transport map between $\mu$ and $T_{\#} \mu$, it follows that $\Pi$ is an optimal coupling. Applying the Euclidean version of Theorem \ref{Ambrosio-Pratelli} we conclude that the support of $\Pi$ is cyclically monotone. 
In particular, we have the inequality 
$$ \sum_{i=1}^Nd_E^p(\pi(q_i)), T(\pi(q_i)))\leq \sum_{i=1}^N d_E^p(\pi(q_{i+1}), T(\pi(q_i))).$$

We continue by noticing that by the very definition of $\widehat{T}$ (see \eqref{lifted-map}) and the formula of the Heisenberg-Kor\'anyi metric, we have the equality:
$$ \sum_{i=1}^N d_E^p(\pi(q_i), T(\pi(q_i))= \sum_{i=1}^N d_H^p(q_i, \widehat{T}(q_i)).$$

On the other hand, we notice also that by the definition of the Heisenberg-Kor\'anyi metric, for any $q, q'\in \HH^n$ the following inequality holds:
$$ d_H^p(q, \widehat{T}(q')) \geq d_E^p(\pi(q), \pi(\widehat{T}(q'))) = d_E^p(\pi(q), T(\pi(q'))).$$

Combining these relations we conclude
\begin{equation*}
\begin{split}
\sum_{i=1}^N d_H^p(q_i, \widehat{T}(q_i))&= \sum_{i=1}^N d_E^p(\pi(q_i), T(\pi(q_i))\\
&\leq \sum_{i=1}^N d_E^p(\pi(q_{i+1}), T(\pi(q_i)))\\
&\leq \sum_{i=1}^N d_H^p(q_{i+1}, \widehat{T}(q_i)),
\end{split}
\end{equation*}
proving the cyclical monotonicity of the support of $\widehat{\Pi}$ and finishing the proof of the lemma. 
\end{proof}

We aim to apply the above result to produce some optimal mass transport maps that are useful for our purposes. For a non-zero horizontal vector $U= (u,v, 0)\in \mathbb H^n$ we shall consider the family of {\it right-translations} $\widehat{T}_{tU}: \mathbb{H}^n \to \mathbb{H}^n$ defined by $\widehat{T}_{tU}(q)= q \ast (tU)$ for $t\in \mathbb{R}$. For any fixed $q \in \mathbb{H}^n$ and non-zero horizontal vector $U$ the 
curve $t \mapsto \widehat{T}_{tU}(q)$ is a complete geodesic in $\mathbb{H}^n$. Moreover, all complete geodesics in 
$\mathbb{H}^n$ are of this form (see \cite{BFS} Corollary 3.15). 
We are interested in a similar characterization of complete geodesics in $\mathcal{W}_p(\mathbb H^n)$ as described above. In what follows we shall assume that $p>1$.

Let us note that right-translations have been already indicated as optimal transport maps between absolutely continuous measures in the case $p=2$  in \cite{AR}. Originally this result was formulated by Ambrosio and Rigot in the setting of the Wasserstein space with respect to the Carnot-Carath\'eodory metric. As we shall see below, such horizontal right-translations are optimal transport maps also in our case of the Heisenberg-Kor\'anyi metric for general $p>1$. Furthermore, in analogy with the Euclidean case (see \cite[Proposition 3.6]{K}) complete geodesics are all induced by right-translations.

\begin{proposition}\label{newprop1} 
Let $U=(u,v,0)$ be a horizontal vector in $\mathbb{H}^n$ and $t\in \mathbb{R}$. Then the mapping $\widehat{T}_{tU}: \mathbb{H}^n \to \mathbb{H}^n$ defined by 
 \begin{equation} \label{right-transl}
 \widehat{T}_{tU}(x,y,z)= (x,y,z)\ast (tU), \ \text{for} \ (x,y,z) \in \mathbb{H}^n, 
 \end{equation}
that is a right-translation in $\mathbb{H}^n$ by the element $tU$ becomes an optimal transport map between any $\nu \in \wth$ and its image $(\widehat{T}_{tU})_{\#}\nu$. Moreover, a curve $\gamma: \mathbb{R} \to \mathcal{W}_p(\mathbb H^n)$ is a complete geodesic if and only if there exists
 $\mu \in\mathcal{W}_p(\mathbb H^n)$ and a non-zero horizontal vector $U= (u,v,0)$ such that $\gamma(t)= (\widehat{T}_{tU})_{\#} \mu$. 

\end{proposition}
\begin{proof}
In order to verify the first part of the claim we will apply Lemma \ref{lift} to the mapping 
$$ T_{tU}: \mathbb{R}^{2n} \to \mathbb{R}^{2n} , \ \ T_{tU}(x,y) = (x+tu, y +tv), \ (x,y) \in \mathbb{R}^{2n}. $$
We have to check that $T_{tU}$ is an optimal transport map between an arbitrary 
$\mu \in \cW_{p}(\mathbb{R}^{2n})$ and its push-forward 
$(T_{tU})_{\#}\mu.$ {However, it is known by \cite[Section 5.1 (Appendix)]{PSP15} that translations are optimal transport maps on the Euclidean space $\R^d$ as long as the transport cost is convex, that is, $c(x,y)=l(\norm{x-y})$ for a convex function $l.$}

Now we know that $\widehat{T}_{tU}:  \mathbb{H}^n \to \mathbb{H}^n$ is an optimal transport map between $\mu$ and $(\widehat{T}_{tU})_{\#}\mu$, and thus the curve $\gamma(t): = (\widehat{T}_{tU})_{\#}\mu$ is a complete geodesic in 
$\mathcal{W}_p(\mathbb H^n)$. Indeed, for  $s<t$ we have 
$$d^p_{\mathcal{W}_p}(\gamma(t),\gamma(s))= \int_{\mathbb{H}^n}d^p_H(q, \widehat{T}_{(t-s)U}(q)) ~\mathrm{d}\gamma(s)(q) = (t-s)^p||U||^p.$$
To prove the reverse implication, suppose that $\gamma: \mathbb{R} \to \mathcal{W}_p(\mathbb H^n)$ is a complete geodesic. We shall define $\mu:= \gamma(0)$ and we intend to find a non-zero horizontal vector $U= (u,v,0)$ such that $\gamma(t)= (\widehat{T}_{tU})_{\#} \mu$.
\par
Let $a<b<c \in \R$ and let $\pi_{ab}$ and $\pi_{bc}$ be optimal transport plans from $\gamma(a)$ to $\gamma(b),$ and from $\gamma(b)$ to $\gamma(c),$ respectively. By the ``gluing lemma'', see \cite[Lemma 7.6]{V}, there is a probability measure $\pi_{abc} \in \mathcal{P}((\HH^n)^3)$ such that $(\pi_{abc})_{12}=\pi_{ab}, \, (\pi_{abc})_{23}=\pi_{bc},$ and $\pi_{ac}:=(\pi_{abc})_{13}$ is a coupling of $\gamma(a)$ and $\gamma(c).$ Here, we use the notation 
$(\pi_{abc})_{12}$ for the push-forward of $\pi_{abc}$ by the projection of $(\HH^n)^3$ onto the first two coordinates. A similar meaning is applied to $(\pi_{abc})_{13}$ and $(\pi_{abc})_{23}$ as well.  Using these notations we can write:
\begin{equation*}
\begin{split}
&\dwp(\gamma(a),\gamma(c))\leq\\ 
&\leq \ler{~\iiint\limits_{(\HH^n)^3} d_H^p(q_a,q_c) \dd \pi_{abc}(q_a,q_b,q_c)}^{\frac{1}{p}}\\
&\leq \ler{~\iiint\limits_{(\HH^n)^3} \Big(d_H(q_a,q_b)+d_H(q_b,q_c)\Big)^p \dd \pi_{abc}(q_a,q_b,q_c)}^{\frac{1}{p}}\\
&\leq
\ler{~\iiint\limits_{(\HH^n)^3} d_H^p(q_a,q_b) \dd \pi_{abc}(q_a,q_b,q_c)}^{\frac{1}{p}}
+
\ler{~\iiint\limits_{(\HH^n)^3} d_H^p(q_b,q_c) \dd \pi_{abc}(q_a,q_b,q_c)}^{\frac{1}{p}}\\
&=\dwp(\gamma(a),\gamma(b))+\dwp(\gamma(b),\gamma(c))\\
&=\dwp(\gamma(a),\gamma(c)).
\end{split}
\end{equation*}
This means that all inequalities in the above chain are saturated. The saturation of the first inequality is equivalent to the optimality of the coupling $\pi_{ac}.$ The saturation of the second inequality implies that the triangle inequality $d_H(q_a,q_c) \leq d_H(q_a,q_b)+d_H(q_b,q_c)$ is saturated $\pi_{abc}$-almost everywhere. The assumption $p>1$ becomes crucial when we deduce from {the geodesic property of $\gamma$ and} the saturation of the third inequality, which is an $L^p$-Minkowski inequality, that
\be \label{eq:Lp-mink-sat}
d_H(q_a,q_b)=\frac{b-a}{c-a}d_H(q_a,q_c) \text{ and } d_H(q_b,q_c)=\frac{c-b}{c-a}d_H(q_a,q_c)
\ee
$\pi_{abc}$-almost everywhere. By the horizontal strict convexity of the Heisenberg-Kor\'anyi norm (see \cite{BFS}), this implies that 

for $\pi_{abc}$-almost every $(q_a,q_b,q_c)$ the vectors
$q_a^{-1}*q_b$ and $q_b^{-1}*q_c$ and $q_a^{-1}*q_c$ are horizontal vectors, and
\be \label{eq:HK-hor-str-conv}
q_a^{-1}*q_b=\delta_{\frac{b-a}{c-a}} \ler{q_a^{-1}*q_c} \text{ and } q_b^{-1}*q_c=\delta_{\frac{c-b}{c-a}}\ler{q_a^{-1}*q_c},
\ee
which in turn implies that
\be \label{eq:HK-hor-str-conv-2}
q_a^{-1}*q_b=\delta_{\frac{b-a}{c-b}}\ler{q_b^{-1}*q_c}.
\ee
Let $\pi_{01}$ be an optimal coupling of $\gamma(0)$ and $\gamma(1).$
Let $t>1,$ choose $a:=0, \, b:=1,$ and $c:=t,$ and let $\pi_{1t}$ be an optimal coupling of $\gamma(1)$ and $\gamma(t).$
With the above choice of $a,b,$ and $c,$ the first equation of \eqref{eq:HK-hor-str-conv} reads as
\be \label{eq:master01t}
q_0^{-1}*q_1=\delta_{\frac{1}{t}}\ler{q_0^{-1}*q_t},
\ee
which equivalent to 
\be \label{eq:master01tv2}
q_t=q_0*\delta_{t}\ler{q_0^{-1}*q_1}
\ee
by a straightforward computation and using the fact that the inverse of the non-isotropic dilation $\delta_{\frac{1}{t}}$ is $\delta_t.$ 
So we get that \eqref{eq:master01tv2} holds for $\pi_{01t}$ -a.e. $(q_0,q_1,q_t),$ where $\pi_{01t}$ is the gluing of $\pi_{01}$ and $\pi_{1t},$ which implies that 
\be \label{eq:pi-01t-form}
\pi_{01t}=\ler{\lers{(q_0,q_1) \mapsto \ler{q_0, q_1, q_0*\delta_{t}\ler{q_0^{-1}*q_1}}}}_{\#}\ler{\pi_{01}}.
\ee
In particular,
\be \label{eq:pi-0t-form}
\pi_{0t}=\ler{\pi_{01t}}_{13}
=\ler{\lers{(q_0,q_1) \mapsto \ler{q_0, q_0*\delta_{t}\ler{q_0^{-1}*q_1}}}}_{\#}\ler{\pi_{01}},
\ee
and
\be \label{eq:gamma-t-form}
\gamma(t)=\ler{\lers{(q_0,q_1) \mapsto q_0*\delta_{t}\ler{q_0^{-1}*q_1}}}_{\#}\ler{\pi_{01}}.
\ee
Let $s>0$ and consider the choice $a:=-s, \, b:=0,$ and $c:=1.$ Let $\pi_{-s0}$ be an optimal coupling of $\gamma(-s)$ and $\gamma(0),$ and let $\pi_{-s01}$ be the gluing of $\pi_{-s0}$ and $\pi_{01}.$ In this case, equation \eqref{eq:HK-hor-str-conv-2} has the following form:
\be \label{eq:master-s01}
q_{-s}^{-1}*q_0=\delta_{s}\ler{q_{0}^{-1}*q_1},
\ee
which is equivalent (by a short algebraic computation on the Heisenberg group) to
\be \label{eq:master-s01v2}
q_{-s}=q_0*\delta_{s}\ler{q_1^{-1}*q_0}.
\ee
That is, \eqref{eq:master-s01v2} holds for $\pi_{-s01}$-a.e. $(q_{-s},q_0,q_1),$ and hence it follows that 
\be \label{eq:pi-mins01-form}
\pi_{-s01}=\ler{\lers{(q_0,q_1) \mapsto \ler{q_0*\delta_{s}\ler{q_1^{-1}*q_0},q_0, q_1,}}}_{\#}\ler{\pi_{01}}.
\ee
In particular,
\be \label{eq:pi-mins0-form}
\pi_{-s0}=\ler{\pi_{-s01}}_{12}
=\ler{\lers{(q_0,q_1) \mapsto \ler{q_0*\delta_{s}\ler{q_1^{-1}*q_0},q_0}}}_{\#}\ler{\pi_{01}},
\ee
and
\be \label{eq:gamma-mins-form}
\gamma(-s)=\ler{\lers{(q_0,q_1) \mapsto q_0*\delta_{s}\ler{q_1^{-1}*q_0}}}_{\#}\ler{\pi_{01}}.
\ee
In the next step of the proof we intend to extract additional information on the form of the coupling $\pi_{01}$ appearing on the right side of \eqref{eq:gamma-t-form-v2} and \eqref{eq:gamma-mins-form-v2}. We shall prove that there exists a non-zero horizontal vector $U$ with the property that 
$$
\pi_{01}=\ler{\mathrm{id} \times \hat{T}_{U}}_{\#}\ler{\gamma(0)}.
$$
In order to do that let us note first, by the comparison of \eqref{eq:pi-0t-form} and \eqref{eq:pi-mins0-form} that the gluing of $\pi_{-s,0}$ and $\pi_{0,t},$ which we shall denote by $\pi_{-s0t},$ is given by
\be \label{eq:pi-mins0t-form}
\pi_{-s0t}
=\ler{\lers{(q_0,q_1) \mapsto \ler{q_0*\delta_{s}\ler{q_1^{-1}*q_0},q_0, q_0 * \delta_t\ler{q_0^{-1}*q_1}}}}_{\#}\ler{\pi_{01}}.
\ee
So if $(q_0,q_1) \in \supp\ler{\pi_{01}}$ then 
$$
\ler{q_0*\delta_{s}\ler{q_1^{-1}*q_0},q_0 * \delta_t\ler{q_0^{-1}*q_1}} \in \supp\ler{\pi_{-st}}
$$
where $\pi_{-st}=\ler{\pi_{-s0t}}_{13}.$ As we obtained $\pi_{-st}$ by gluing the optimal couplings $\pi_{-s0}$ (between $\gamma(-s)$ and $\gamma(0)$) and $\pi_{0t}$ (between $\gamma(0)$ and $\gamma(t)$), the coupling $\pi_{-st}$ is also optimal between $\gamma(-s)$ and $\gamma(t).$ Therefore, its support is cyclically monotone, which is going to be an essential point of the proof.
\par
Let $(q_0,q_1)$ and $(q_0',q_1')$ be points in the support of $\pi_{01},$ and let $(u,v,0):=q_0^{-1}*q_1$ and $(u',v',0):=q_0'^{-1}*q_1'$ -- note that this parametrization relies on the information we obtained before that both $q_0^{-1}*q_1$ and $q_0'^{-1}*q_1'$ are horizontal vectors. Let us choose $s:=t.$
Applying the cyclical monotonicity to the points $(q_{-t}, q_t)$ and $(q'_{-t}, q'_t)$ in the support of $\pi_{-tt}$ we obtain: 
\begin{equation} \label{eq:cyclical-mon} 
d^p_H(q_{-t}, q_t) + d^p_H(q'_{-t}, q'_{t}) \leq d^p_H(q_{-t}, q'_t) + d^p_H(q'_{-t}, q_t).
\end{equation}
If two points in $\HH^n$ lie on the same horizontal line then their Heisenberg distance coincides with the Euclidean distance. Therefore, the left-hand side of \eqref{eq:cyclical-mon} can be calculated as follows:
$$
\norm{q_{-t}^{-1}*q_t}_H^p+\norm{q_{-t}'^{-1}*q'_t}_H^p
=\norm{(2tu,2tv,0)}_H^p+\norm{(2tu',2tv',0)}_H^p
$$
\be \label{eq:cycl-mon-lhs}
=(2t)^p\ler{\norm{(u,v,0)}_{H}^p+\norm{(u',v',0)}_{H}^p}.
\ee
The right-hand side of \eqref{eq:cyclical-mon} is more involved and it reads as follows:
$$
\norm{q_{-t}^{-1}*q_t'}_H^p+\norm{q_{-t}'^{-1}*q_t}_H^p
=\norm{\ler{q_0*\delta_{t}\ler{q_1^{-1}*q_0}}^{-1}*\ler{q_0'*\delta_{t}\ler{q_0'^{-1}*q_1'}}}_H^p+
$$
$$
+\norm{\ler{q_0'*\delta_{t}\ler{q_1'^{-1}*q_0'}}^{-1}*\ler{q_0*\delta_{t}\ler{q_0^{-1}*q_1}}}_H^p
$$
$$
=\norm{(tu,tv,0)*q_0^{-1}*q_0'*(tu',tv',0)}_H^p+\norm{(tu',tv',0)*q_0'^{-1}*q_0*(tu,tv,0)}_H^p.
$$
Let us introduce $(x,y,z):=q_0^{-1}*q_0'.$ Then
$$
q_{-t}^{-1}*q_t'=(tu,tv,0)*(x,y,z)*(tu',tv',0)
$$
$$
=\ler{x+t(u+u'),y+t(v+v'), z+2t(v\cdot x - u \cdot y+ u' \cdot y -v'\cdot x)+2t^2(u'\cdot v-v'\cdot u)}.
$$
Therefore, for large values of $t > 1$ we have that
\begin{equation*}
\begin{split}
\norm{q_{-t}^{-1}*q_t'}_H^p&=\Bigg(\ler{\abs{x+t(u+u')}^2+\abs{y+t(v+v')}^2}^2+\\
&\hspace{1.5cm}+\ler{z+2t(v\cdot x - u \cdot y+ u' \cdot y -v'\cdot x)+2t^2(u'\cdot v-v'\cdot u)}^2\Bigg)^{\frac{p}{4}}\\
&=\ler{\ler{t^2\abs{u+u'}^2+t^2\abs{v+v'}^2}^2+4t^4(u'\cdot v-v'\cdot u)^2+\bigO(t^3)}^{\frac{p}{4}}.
\end{split}
\end{equation*}
A very similar computation shows that 
\begin{equation*}
\begin{split}
\norm{q_{-t}'^{-1}*q_t}_H^p
&=\norm{(tu',tv',0)*(-x,-y,-z)*(tu,tv,0)}_H^p\\
&=\ler{\ler{t^2\abs{u'+u}^2+t^2\abs{v'+v}^2}^2+4t^4(u\cdot v'-v\cdot u')^2+\bigO(t^3)}^{\frac{p}{4}}.
\end{split}
\end{equation*}
Therefore,
$$
\lim_{t \to +\infty} t^{-p} \norm{q_{-t}^{-1}*q_t'}_H^p
=\ler{\ler{\abs{u+u'}^2+\abs{v+v'}^2}^2+\ler{2(u'\cdot v-v'\cdot u)}^2}^{\frac{p}{4}}
$$
\be \label{eq:t-lim-rhs-1}
=\norm{(u,v,0)*(u',v',0)}_H^p.
\ee
Similarly,
$$
\lim_{t \to +\infty} t^{-p} \norm{q_{-t}'^{-1}*q_t}_H^p
=\ler{\ler{\abs{u'+u}^2+\abs{v'+v}^2}^2+\ler{2(u\cdot v'-v\cdot u')}^2}^{\frac{p}{4}}
$$
\be \label{eq:t-lim-rhs-2}
=\norm{(u',v',0)*(u,v,0)}_H^p.
\ee
We now compare \eqref{eq:cycl-mon-lhs} with \eqref{eq:t-lim-rhs-1} and \eqref{eq:t-lim-rhs-2} to obtain that the cyclical monotonicity of the support of $\pi_{-tt}$ implies that
\be \label{eq:CM-conclusion}
2^p\ler{\norm{(u,v,0)}_H^p+\norm{(u',v',0)}_H^p} \leq \norm{(u,v,0)*(u',v',0)}_H^p+\norm{(u',v',0)*(u,v,0)}_H^p.
\ee
The Heisenberg norm is subadditive with respect to the group operation $*$ and hence it follows from \eqref{eq:CM-conclusion} that 
\be \label{eq:CM-concl-cor}
2^p\ler{\norm{(u,v,0)}_H^p+\norm{(u',v',0)}_H^p} \leq 2 \ler{\norm{(u,v,0)}_H+\norm{(u',v',0)}_H}^p.
\ee
Straightforward calculations show that \eqref{eq:CM-concl-cor} is equivalent to 
\be \label{eq:m-p-m-1}
\ler{\frac{\norm{(u,v,0)}_H^p+\norm{(u',v',0)}_H^p}{2}}^{\frac{1}{p}}\leq
\frac{\norm{(u,v,0)}_H+\norm{(u',v',0)}_H}{2},
\ee
that is, the $p$-power mean of $\norm{(u,v,0)}_H$ and $\norm{(u',v',0)}_H$ is bounded from above by their arithmetic mean. On the other hand, the $p$-power mean of non-negative numbers is strictly monotone increasing in $p,$ that is, 
$$
\ler{\frac{\norm{(u,v,0)}_H^p+\norm{(u',v',0)}_H^p}{2}}^{\frac{1}{p}}\geq
\frac{\norm{(u,v,0)}_H+\norm{(u',v',0)}_H}{2},
$$
which implies that the two sides of \eqref{eq:m-p-m-1} are equal. Consequently, the two sides of \eqref{eq:CM-concl-cor} are also equal, which implies that
$$
\norm{(u,v,0)*(u',v',0)}_H=\norm{(u,v,0)}_H+\norm{(u',v',0)}_H.
$$
Using again the Horizontal strict convexity of the Heisenberg norm \cite{BFS} we get that there is a $\lambda \geq 0$ such that $(u',v',0)=\lambda (u,v,0).$ Furtheremore, the inequality \eqref{eq:m-p-m-1} can happen only if $\norm{(u,v,0)}_H=\norm{(u',v',0)}_H.$ Therefore, $\lambda=1$ and $(u',v',0)=(u,v,0).$ That is, we obtained that $q_0^{-1} *q_1=q_0'^{-1}*q_1'.$ This means that there is a horizontal vector $U=(u,v,0)$ such that $q_1=q_0*U$ for all points $(q_0,q_1)$ in the support of $\pi_{01}.$ This means that
$$
\pi_{01}=\ler{\mathrm{id} \times \hat{T}_U}_{\#} \ler{\gamma(0)}
$$
and in particular $\gamma(1)=\ler{\hat{T}_U}_{\#}\ler{\gamma(0)}.$ Using the equations \eqref{eq:gamma-t-form} and \eqref{eq:gamma-mins-form} we get that
\be \label{eq:gamma-t-form-v2}
\gamma(t)=\ler{\hat{T}_{tU}}_{\#}\ler{\gamma(0)}.
\ee
and
\be \label{eq:gamma-mins-form-v2}
\gamma(-s)=\ler{\hat{T}_{(-sU)}}_{\#}\ler{\gamma(0)}
\ee
for any $s>0$ and $t>1.$
If $0<r<1,$ then choosing $a:=0, b:=r,$ and $c:=1$ we get that for any optimal coupling $\pi_{0r}$ of $\gamma(0)$ and $\gamma(r)$ and for any optimal coupling $\pi_{r1}$ of $\gamma(r)$ and $\gamma(1)$ we have
\be \label{eq:pi-0r1-master}
q_0^{-1}*q_r= \delta_r\ler{q_0^{-1}*q_1} \Leftrightarrow  q_r= q_0 *\delta_r\ler{q_0^{-1}*q_1}
\ee
for $\pi_{0r1}$-a.e. $(q_0,q_r,q_1),$ where $\pi_{0r1}$ is the gluing of $\pi_{0r}$ and $\pi_{r1}.$ The coupling $\ler{\pi_{0r1}}_{13}$ is optimal for $\gamma(0)$ and $\gamma(1)$ and hence $q_0^{-1}*q_1\equiv U$ for $\ler{\pi_{0r1}}$-a.e. $(q_0,q_1).$ This implies by \eqref{eq:pi-0r1-master} that $q_r=q_0*(rU)$ $\pi_{0r1}$-almost everywhere, and hence in particular
\be \label{eq:gamma-r-form}
\gamma(r)=\ler{\pi_{0r1}}_2=\ler{\lers{q_0 \mapsto q_0*(rU)}}_{\#}\ler{\gamma(0)}=\ler{\hat{T}_{(rU)}}_{\#}\ler{\gamma(0)}
\ee
as desired.

\color{black}
\end{proof}

\begin{remark}\label{p=1 geod} The assumption that $p>1$ is crucial in Proposition \ref{newprop1}. If $p=1$, then complete geodesics can have a more complicated structure. To see a very simple example, the complete geodesic $\gamma:\mathbb{R} \to \mathcal{W}_1(\HH)$, $\gamma(t):=\delta_{(t,0,0)}$ can be modified as 
    \[   
\widetilde{\gamma}(t):= 
     \begin{cases}
       \delta_{(t,0,0)} &\quad\text{if}\quad t\notin[0,1], \\
       (1-t)\delta_{(0,0,0)}+t\delta_{(1,0,0)}&\quad\text{if}\quad t\in[0,1], \\ 
     \end{cases}
\]
which is still a complete geodesic, but there is no $\mu\in\mathcal{W}_1(\HH)$ and horizontal vector $U$ such that $\gamma(t)= (\widehat{T}_{tU})_{\#} \mu$.
\end{remark}
Our next proposition is concerned with horizontal dilations. This will be used to describe geodesic rays in $\mathcal{W}_p(\mathbb{H}^n)$. Let us recall that geodesic rays in Euclidean Wasserstein spaces $\mathcal{W}_p(\mathbb{R}^n)$  can be obtained by considering Euclidean dilations
$$D_{\lambda}:\mathbb{R}^n\to\mathbb{R}^n,\quad D_{\lambda}(x)=\lambda x$$ 
for $\lambda \geq 0$ that are optimal transport maps from any $\mu \in \mathcal{W}_p(\mathbb{R}^n)$ to $(D_{\lambda})_{\#}\mu$ (see Section 2.3 in \cite{K}, for the case $p=2$, and the Appendix of \cite{GTV2} for the case of general $p$). Unfortunately in the case of the Heisenberg group,  the non-isotropic Heisenberg dilations are not optimal transport maps and therefore we cannot obtain geodesic rays in $\mathcal{W}_p(\mathbb{H}^n)$ in this way. 

To get around this difficulty we shall consider {\it horizontal dilations}: 
$\widehat{D}_{\lambda}: \mathbb{H}^n \to \mathbb{H}^n$ defined by 
$$\widehat{D}_{\lambda} (x,y,z) = (\lambda x, \lambda y, z), \ \text{for all} \ \lambda \geq 0 , \ \text{and} \ (x,y,z) \in \mathbb{H}^n.$$

\begin{proposition} \label{horizontal-dilations}
{Let us fix a $p>1$, and c}onsider the horizontal dilation $\widehat{D}_{\lambda}: \mathbb{H}^n \to \mathbb{H}^n$ for all $\lambda \geq 0$. We claim that $\widehat{D}_{\lambda}$ is an optimal transport map between $\mu$ and $(\widehat{D}_{\lambda})_{\#}\mu.$
Moreover, if $\mu$ is not supported on the vertical $0z$-axis, then the curve 
$$\gamma:[0,\infty)\to\Wphn;\quad\gamma(\lambda):=(\widehat{D}_{\lambda})_{\#}\mu$$
is a geodesic ray containing $\mu=\gamma(1)$.
\end{proposition}

\begin{proof} To verify the first claim of the proposition, we shall apply again Lemma \ref{lift} to the Euclidean dilation  $D_{\lambda}: \mathbb{R}^{2n} \to \mathbb{R}^{2n}$ given by $D_{\lambda} (x,y) = (\lambda x, \lambda y)$.  

Therefore, we have to prove first that the Euclidean dilation $D_\lambda$ is an optimal transport map between $\mu$ and $(D_\lambda)_{\#} \mu$ for any $\mu \in \cW_p(\R^{2n}).$ We shall use \cite[Theorem 5.10]{Villani} by defining functions $\psi, \phi$ on $\R^{2n}$ {such that $\phi(x',y')-\psi(x,y)\leq \norm{(x,y)-(x',y')}^p$ for every $(x,y),(x',y') \in \R^{2n},$ with equality whenever $(x',y')=(\lambda x,\lambda y)$}. By symmetry, it is sufficient to prove the optimality of $D_\lambda$ for $\lambda \geq 1.$ 
Let
$$
\psi((x,y)):=(\lambda-1)^{p-1}\norm{(x,y)}^p \qquad ((x,y) \in \R^{2n})
$$
and
$$
\phi((x',y')):=\inf \lers{\psi((x,y))+c((x,y),(x',y')) \, \middle| \, (x,y) \in \R^{2n}}=
$$
$$
=\inf \lers{(\lambda-1)^{p-1}\norm{(x,y)}^p+\norm{(x,y)-(x',y')}^p \, \middle| \, (x,y) \in \R^{2n}}.
$$
One-variable calculus shows that for fixed $(x',y') \in \R^{2n},$ the unique minimizer of the strictly convex function
$$
(x,y) \mapsto (\lambda-1)^{p-1}\norm{(x,y)}^p+\norm{(x,y)-(x',y')}^p
$$
is $(x_0,y_0)=(x',y') /\lambda.$
Therefore the inequality
\begin{equation*}
\begin{split}
\phi((x',y'))-\psi((x,y))&=\min_{(\tilde{x},\tilde{y})} \lers{\psi((\tilde{x},\tilde{y}))+c((\tilde{x},\tilde{y}),(x',y'))}-\psi((x,y))\\
&\leq c((x,y),(x',y')),
\end{split}
\end{equation*}
which is obviously true for all $(x,y), (x',y') \in \R^{2n},$ is saturated whenever $(x',y')=(\lambda x, \lambda y),$ which proves the optimality of $D_\lambda.$

 So we can apply Lemma \ref{lift} to conclude that $\widehat{D}_{\lambda}: \mathbb{H}^n \to \mathbb{H}^n$ is an optimal transport map from $\mu$ to the measure $(\widehat{D}_{\lambda})_{\#}\mu$. This implies that if the measure $\mu$ is not supported on the vertical $0z$-axis, then the curve $\lambda \mapsto \gamma(\lambda) :=(\widehat{D}_{\lambda})_{\#}\mu$ is a geodesic ray containing $\mu= \gamma(1)$. Indeed, to see this, let us consider $0< \lambda_1 < \lambda_2 $. Note that the optimal transport map from $\gamma(\lambda_1)=(\widehat{D}_{\lambda_1})_{\#}\mu $ to $\gamma(\lambda_2)= (\widehat{D}_{\lambda_2})_{\#}\mu$ will be $\widehat{D}_{\frac{\lambda_2}{\lambda_1}}$. We can therefore calculate:
 \begin{equation} \label{HD}
 d^p_{\mathcal{W}_p}(\gamma(\lambda_1), \gamma(\lambda_2)) = \int_{\mathbb{H}^n}d^p_H(q, \widehat{D}_{\frac{\lambda_2}{\lambda_1}}(q)) ~\mathrm{d}\gamma(\lambda_1)(q). 
 \end{equation}
 
 Making the change of variables $q = \widehat{D}_{\lambda_1}(q')$ and using the fact that $\widehat{D}_{\frac{\lambda_2}{\lambda_1}} \circ \widehat{D}_{\lambda_1} = \widehat{D}_{\lambda_2} $ we obtain that the right-hand side of \eqref{HD} will become 
 \begin{equation}
 \begin{split}
 \int_{\mathbb{H}^n}d^p_H(q, \widehat{D}_{\frac{\lambda_2}{\lambda_1}}(q)) ~\mathrm{d}\gamma(\lambda_1)(q)&= 
 \int_{\mathbb{H}^n}d^p_H (\widehat{D}_{\lambda_1}(q'), \widehat{D}_{\lambda_2}(q') )~\mathrm{d}\mu(q')\\
 &= |\lambda_1-\lambda_2|^p \int_{\mathbb{H}^n}||(x,y)||^p ~\mathrm{d}\mu. 
 \end{split}
 \end{equation}
 
 Combining the above relations we obtain:
 
 \begin{equation} \label{geod-ray}
 d^p_{\mathcal{W}_p}(\gamma(\lambda_1), \gamma(\lambda_2)) =|\lambda_1-\lambda_2|^p\int_{\mathbb{H}^n}||(x,y)||^p ~\mathrm{d}\mu. 
 \end{equation}
 
 Let us observe that if $\mu$ is supported on the $0z$-axis, then the integral on the right-hand side vanishes. In this case, $\gamma$ degenerates to a single point. On the other hand, if the support of $\mu$ is not contained in the $0z$-axis then \eqref{geod-ray} can be written equivalently as
  \begin{equation}
 d_{\mathcal{W}_p}(\gamma(\lambda_1), \gamma(\lambda_2)) = C |\lambda_1-\lambda_2|, 
 \end{equation}
 where $C= \Big(\int_{\mathbb{H}^n}||(x,y)||^p ~\mathrm{d}\mu\Big)^{\frac{1}{p}} >0$. 
 \end{proof}

 Let us notice that the endpoint of the above geodesic ray $\gamma(0)$ is vertically supported as its support is contained in the $0z$-axis. The following proposition gives a characterization of vertically supported measures as endpoints of geodesic rays. We remark that dilations have been used also in the Euclidean case to understand the action of an isometry on the set of Dirac masses. An important observation in \cite{K} is that general geodesic rays ending at a Dirac mass cannot be extended past it (see Section 2.3 and Section 6.2 in \cite{K}). The following statement is a replacement of this fact in our Heisenberg setting where, roughly speaking, we replace Dirac masses by vertically supported measures. 

\begin{proposition} \label{vertically-supported}
Let $\mu \in \mathcal{W}_p(\mathbb H^n)$. Then $\mu$ is vertically supported if and only if it has the property that if $\mu$ is contained in a geodesic ray $\gamma$, then either $\mu$ is the endpoint of $\gamma$ or $\gamma$ can be extended to a complete geodesic. 
\end{proposition}

\begin{proof} 
To prove one of the implications from the statement, assume that  $\mu \in \mathcal{W}_p(\mathbb{H}^n)$ has the property that if it lies on a geodesic ray, then it is either its endpoint or if not, then the geodesic ray can be extended to a complete geodesic. Assume by contradiction that $\mu$ is not vertically supported. For any $t\geq 0$ we consider the horizontal dilations: 
$\widehat{D}_t: \mathbb{H}^n \to \mathbb{H}^n$ defined by $\widehat{D}_t(x,y,z) = (tx,ty, z)$. By applying Proposition \ref{horizontal-dilations} we see that the curve $t \mapsto \gamma(t) :=(\widehat{D}_t)_{\#}\mu$ is a geodesic ray containing $\mu= \gamma(1)$. Since $\mu$ is not vertically supported, it follows that it cannot coincide with the endpoint $\gamma(0)$, which is vertically supported by construction.  According to our assumption on $\mu$, the geodesic ray $\gamma$ can be extended to a complete geodesic. On the other hand, by Proposition \ref{newprop1} the form of complete geodesics implies that if one 
of the measures on the geodesic is vertically supported then all of them must be so. Since $\mu$ is not vertically supported, this gives the desired contradiction concluding the proof of one of the implications. 

To prove the other implication assume that $\mu$ is vertically supported 

and $\mu=\gamma(0)$ for a geodesic ray $\gamma: [0,\infty) \to \Wphn.$ Assume that this $\gamma$ can be extended to an interval $(-a, \infty)$ for some $a>0,$ and let $s\in(0,a).$ We shall proceed in the way similar to the proof of Proposition \ref{newprop1}, and hence we borrow some of the notation from there. Let $(q_0,q_1)$ and $(q_0',q_1')$ be points of the support of $\pi_{01}$ which is an optimal transport plan between $\gamma(0)$ and $\gamma(1),$ let $(u,v,0):=q_0^{-1}*q_1$ and $(u',v',0):=q_0'^{-1}*q_1'$ and let $q_s,q_t, q_s',q_t'$ be defined as in the proof of Proposition \ref{newprop1}. Note that $q_0^{-1}*q_0'$ is a vertical vector as $\mu$ is vertically supported and hence let us denote it by $(0,0,z).$
Let $t>0$ be arbitrary (we will focus on the $t \gg 1$ regime) and let $s \in (0,a)$ be arbitrary but fixed. The cyclical monotonicity of the support of $\pi_{-s,t}$ reads as follows:
\begin{equation} \label{eq:cyclical-mon-ver2} 
d^p_H(q_{-s}, q_t) + d^p_H(q'_{-s}, q'_{t}) \leq d^p_H(q_{-s}, q'_t) + d^p_H(q'_{-s}, q_t).
\end{equation}
The left-hand side of \eqref{eq:cyclical-mon-ver2} is
\be \label{eq:CM-ver2-LHS}
(s+t)^p \ler{\norm{(u,v,0)}_H^p+\norm{(u',v',0)}_H^p},
\ee
while the right-hand side of \eqref{eq:cyclical-mon-ver2} is equal to
\begin{equation*}
\begin{split}
\|q_{-s}^{-1}&* q_t'\|_H^p+\|q_{-s}'^{-1}* q_t\|_H^p=\\
&=\|(su,sv,0)*q_0^{-1}*q_0'*(tu',tv',0)\|_H^p
+\norm{(su',sv',0)*q_0'^{-1}*q_0*(tu,tv,0)}_H^p\\
&=\norm{(su+tu',sv+tv', 2st (v\cdot u'-u\cdot v')+z)}_H^p+\\
&\hspace{1.0cm}+\norm{(su'+tu,sv'+tv, 2st (v'\cdot u-u'\cdot v)-z)}_H^p\\
&=\ler{\ler{\abs{su+tu'}^2+\abs{sv+tv'}^2}^2+\ler{2st(v\cdot u'-u\cdot v')+z}^2}^{\frac{p}{4}}+
\end{split}
\end{equation*}
\begin{equation*}
\begin{split}
&\hspace{1.6cm}+\ler{\ler{\abs{su'+tu}^2+\abs{sv'+tv}^2}^2+\ler{2st(v'\cdot u-u'\cdot v)-z}^2}^{\frac{p}{4}}\\
&\hspace{0.8cm}=\ler{\ler{t^2 \ler{\abs{u'}^2+\abs{v'}^2}+ 2st(u\cdot u'+v\cdot v')+s^2\ler{\abs{u}^2+\abs{v}^2}}^2+\bigO(t^2)}^{\frac{p}{4}}+
\end{split}
\end{equation*}
\be \label{eq:CM-ver2-RHS}
\hspace{2.4cm}+\ler{\ler{t^2 \ler{\abs{u}^2+\abs{v}^2}+ 2st(u'\cdot u+v'\cdot v)+s^2\ler{\abs{u'}^2+\abs{v'}^2}}^2+\bigO(t^2)}^{\frac{p}{4}}.
\ee
Let $(LHS)$ denote the left-hand side of the cyclical monotonicity inequality \eqref{eq:cyclical-mon-ver2} which is computed in \eqref{eq:CM-ver2-LHS}, and let $(RHS)$ denote the right-hand side of \eqref{eq:cyclical-mon-ver2} computed in \eqref{eq:CM-ver2-RHS}. With this notation one gets
\begin{equation*}
\begin{split}
t^{-p}\ler{RHS}&=\ler{\ler{\abs{u'}^2+\abs{v'}^2}^2+4s\ler{\abs{u'}^2+\abs{v'}^2}(u\cdot u'+v \cdot v')\frac{1}{t}+\bigO\ler{\frac{1}{t^2}}}^{\frac{p}{4}}+\\
&\hspace{0.8cm}+\ler{\ler{\abs{u}^2+\abs{v}^2}^2+4s\ler{\abs{u}^2+\abs{v}^2}(u'\cdot u+v' \cdot v)\frac{1}{t}+\bigO\ler{\frac{1}{t^2}}}^{\frac{p}{4}}\\
&=\abs{(u',v')}^p\ler{1+4s \frac{\langle(u,v), (u',v')\rangle}{\abs{(u',v')}^2} \frac{1}{t}+\bigO\ler{\frac{1}{t^2}}}^{\frac{p}{4}}+\\
&\hspace{0.8cm}+\abs{(u,v)}^p\ler{1+4s \frac{\langle(u',v'), (u,v)\rangle}{\abs{(u,v)}^2} \frac{1}{t}+\bigO\ler{\frac{1}{t^2}}}^{\frac{p}{4}}\\
&=\abs{(u',v')}^p+ps\abs{(u',v')}^{p-2}\langle(u,v), (u',v')\rangle \frac{1}{t}+\bigO\ler{\frac{1}{t^2}}+
\end{split}
\end{equation*}
\be \label{eq:rhs-normalized}
\hspace{1.0cm}+
\abs{(u,v)}^p+ps\abs{(u,v)}^{p-2}\langle(u',v'), (u,v)\rangle \frac{1}{t}+\bigO\ler{\frac{1}{t^2}}
\ee
where we used the binomial expansion $(1+x)^\alpha=1+\alpha x+ \bigO(x^2).$ Furthermore,
\begin{equation*}
\begin{split}
t^{-p}(LHS)&=t^{-p}\ler{(s+t)^4\ler{\abs{u}^2+\abs{v}^2}^2}^{\frac{p}{4}}+t^{-p}\ler{(s+t)^4\ler{\abs{u'}^2+\abs{v'}^2}^2}^{\frac{p}{4}}\\
&=\ler{\ler{\abs{u}^2+\abs{v}^2}^2+4s \ler{\abs{u}^2+\abs{v}^2}^2\frac{1}{t}+\bigO\ler{\frac{1}{t^2}}}^{\frac{p}{4}}+\\
&\hspace{1.0cm}+\ler{\ler{\abs{u'}^2+\abs{v'}^2}^2+4s \ler{\abs{u'}^2+\abs{v'}^2}^2\frac{1}{t}+\bigO\ler{\frac{1}{t^2}}}^{\frac{p}{4}}\\
&=\abs{(u,v)}^p\ler{1+4s\frac{1}{t}+\bigO\ler{\frac{1}{t^2}}}^{\frac{p}{4}}
+\abs{(u',v')}^p\ler{1+4s\frac{1}{t}+\bigO\ler{\frac{1}{t^2}}}^{\frac{p}{4}}
\end{split}
\end{equation*}
\be \label{eq:lhs-normalized}
\hspace{0.7cm}=\abs{(u,v)}^p+ps\abs{(u,v)}^p\frac{1}{t}+\abs{(u',v')}^p+ps\abs{(u',v')}^p\frac{1}{t}+\bigO\ler{\frac{1}{t^2}}.
\ee
By the cyclical monotonicity \eqref{eq:cyclical-mon-ver2}, the expression appearing in \eqref{eq:rhs-normalized} must be bounded from below by the expression appearing in \eqref{eq:lhs-normalized} for every $t \in (0, \infty),$ which implies that 
\be \label{eq:CM-v2-corr}
\abs{(u,v)}^p+\abs{(u',v')}^p
\leq
\abs{(u',v')}^{p-2}\langle(u,v), (u',v')\rangle+\abs{(u,v)}^{p-2}\langle(u',v'), (u,v)\rangle.
\ee
However, by the Cauchy-Schwarz inequality,
$$
\abs{(u',v')}^{p-2}\langle(u,v), (u',v')\rangle+\abs{(u,v)}^{p-2}\langle(u',v'), (u,v)\rangle 
$$
\be \label{eq:Cauchy-Schwarz}
\leq \abs{(u',v')}^{p-1}\abs{(u,v)}+\abs{(u',v')}\abs{(u,v)}^{p-1}.
\ee
Furthermore, recall that for $\lambda \in [0,1]$ the $\lambda$-weighted arithmetic mean of the non-negative numbers $\alpha, \beta \in [0, \infty)$ is given by $\alpha \nabla_{\lambda} \beta=(1-\lambda)\alpha+\lambda \beta$ while the $\lambda$-weighted geometric mean of them is given by $\alpha \#_{\lambda}\beta=\alpha^{1-\lambda}\beta^{\lambda}.$ The arithmetic-geometric mean inequality $\alpha\nabla_{\lambda} \beta \geq \alpha \#_{\lambda} \beta$ always holds true, and it is saturated if and only if $\alpha=\beta$ or $\lambda \in \{0,1\}.$
Therefore,
$$
\abs{(u',v')}^{p-1}\abs{(u,v)}+\abs{(u',v')}\abs{(u,v)}^{p-1}
=\abs{(u',v')}^p \#_{\frac{1}{p}}\abs{(u,v)}^p+\abs{(u',v')}^p \#_{\frac{p-1}{p}}\abs{(u,v)}^p
$$
\be \label{eq:AGM-ineq}
\leq \abs{(u',v')}^p \nabla_{\frac{1}{p}}\abs{(u,v)}^p+\abs{(u',v')}^p \nabla_{\frac{p-1}{p}}\abs{(u,v)}^p=\abs{(u',v')}^p+\abs{(u,v)}^p.
\ee
Consequently, in the chain of inequalities \eqref{eq:CM-v2-corr}, \eqref{eq:Cauchy-Schwarz} and \eqref{eq:AGM-ineq}, every inequality is saturated. In particular, the saturation of the arithmetic-geometric mean inequality implies that $\abs{(u',v')}=\abs{(u,v)},$ and the saturation of the Cauchy-Schwarz inequality implies that $(u',v')$ is a nonnegative scalar multiple of $(u,v).$ Alltogether this means that $(u',v')=(u,v).$
\par
That is, $q_0^{-1}*q_1=q_0'^{-1}*q_1'$ and hence $q_0^{-1}*q_1$ is a constant vector when $(q_0,q_1)$ runs over the support of $\pi_{01}$ (let us denote this constant vector by $U$), which implies that $\gamma(1)=\ler{\widehat{T}_U}_{\#}\mu$ and $\gamma(t)=\ler{\widehat{T}_{tU}}_{\#}\mu$ for all $t\in (-a,\infty).$
\color{black}
But this means that this geodesic ray can be extended to the complete geodesic by the same formula for the range of the parameter $t\in (-\infty, \infty)$. This proves the other implication of Proposition \ref{vertically-supported}.
\end{proof}


\section{The metric rank of $\mathcal{W}_p(\mathbb{H}^n)$: Proof of Theorem 1.1.} The goal of this section is to determine the rank of $\Wphn$, that is, the largest $k$ such that $\mathbb{R}^k$ can be embedded isometrically into $\Wphn$. As stated in Theorem \ref{embprop} the rank of $\Wphn$ is equal to $n$. 

Assume first that $k\leq n$. Since the map $\phi:\mathbb{R}^n\to\HH^n$ defined by $\phi(x):=(x,0,0)$ is an isometric embedding, we obtain that $\Phi:\mathbb{R}^n\to\Wphn$ defined by $\Phi(x):=\delta_{\phi(x)}$ is an isometric embedding as well. If $k\leq n$ then $\mathbb{R}^k$ embeds isometrically into $\mathbb{R}^n$, and thus the composition of the two embeddings provides an isometric embedding of $\mathbb{R}^k$ into $\Wphn$.

To see that the rank of $\Wphn$ cannot be larger than $n$, assume by contradiction, that $\Phi:\mathbb{R}^{n+1}\to\Wphn$ is an isometric embedding, and set $\mu:=\Phi(0)$. For $u\in\mathbb{R}^{n+1}$ consider the complete geodesic line $t\mapsto tu$ $(t\in\mathbb{R})$ that is mapped into a complete geodesic
\begin{equation}
    \gamma:\mathbb{R}\to\Wphn;\qquad \gamma(t)=\Phi(tu).
\end{equation}
By Proposition \ref{newprop1} we know that $\gamma(t)=\widehat{T}_{tU}{}_{\#}\mu$ for some non-zero horizontal vector $U\in\HH^n$.
In this way we can define a mapping \begin{equation}
    \psi:\mathbb{R}^{n+1}\to\HH^n,\qquad\psi(u):=U
\end{equation}
with the property that
\begin{equation}
    \Phi(u)={\widehat{T}_{\psi(u)}}{}_{\#}\mu.
\end{equation} Since $\Phi$ is an isometry, we have
$\dwp({\widehat{T}_{\psi(u)}}{}_{\#}\mu,{\widehat{T}_{\psi(v)}}{}_{\#}\mu)=\|u-v\|$ for all $u,v\in\mathbb{R}^{n+1}$. We claim that $\psi^{-1}(u)\ast\psi(v)$ is a horizontal vector in $\HH^n$. To see this, {let $u\neq v$ and} note that $\Phi(v)={\textcolor{black}{\widehat{T}}_{\psi(v)}}_{\#}\mu$ lies on the infinite geodesic
\begin{equation}
    \widetilde{\gamma}:\mathbb{R}\to\Wphn;\qquad \widetilde{\gamma}(t)=\Phi(u+t(v-u)).
\end{equation}
Applying again Proposition \ref{newprop1} we can conclude that 
\begin{equation}\label{TW1}
    \Phi(v)={\widehat{T}_W}{}_{\#}\Phi(u)
\end{equation}
for some horizontal vector $W\in\HH^n$.
So we can rewrite \eqref{TW1} as
\begin{equation}\label{TW3}
{\widehat{T}_{\psi(v)}}{}_{\#}\mu={\widehat{T}_W}{}_{\#}{\widehat{T}_{\psi(u)}}{}_{\#}\mu={\widehat{T}_{(\psi(u)\ast W)}}{}_{\#}\mu.
\end{equation}
\textcolor{black}{Note that $\mu$ is a probability measure on the Polish (that is, complete and separable) metric space $(\HH^n, d_H).$ Therefore, by Ulam's lemma, $\mu$ is tight (in other words, it vanishes at infinity), which means that for any $\varepsilon>0$ there is a compact set $K_\varepsilon \subset \HH^n$ such that $\mu\ler{K_{\varepsilon}}>1-\varepsilon.$ Consequently, there is no non-trivial right-translation for which $\mu$ is invariant. In particular, if ${\widehat{T}_{\psi(v)}}{}_{\#}\mu={\widehat{T}_{(\psi(u)\ast W)}}{}_{\#}\mu$ as obtained in \eqref{TW3}, then}
$\psi(v)=\psi(u)\ast W$, or equivalently, $\psi^{-1}(u)\ast\psi(v)=W$, is a horizontal vector as we claimed. Let us notice that
$${\widehat{T}_{(\psi^{-1}(u)\ast\psi(v)})}{}_{\#}({\widehat{T}_{\psi(u)}}{}_{\#}\mu)={\widehat{T}_{\psi(v)}}{}_{\#}\mu,$$
and since $\psi^{-1}(u)\ast\psi(v)=W$ is horizontal, we get that the map $$q\mapsto \widehat{T}_{(\psi^{-1}(u)\ast\psi(v))}q$$ is the optimal transport map between ${\widehat{T}_{\psi(u)}}{}_{\#}\mu$ and ${\widehat{T}_{\psi(v)}}{}_{\#}\mu.$ Using this observation we can compute the distance $\dwp^p({\widehat{T}_{\psi(u)}}{}_{\#}\mu,{\widehat{T}_{\psi(v)}}{}_{\#}\mu)$ as follows

\begin{equation}
\begin{split}
\dwp^p({\widehat{T}_{\psi(u)}}{}_{\#}\mu,{\widehat{T}_{\psi(v)}}{}_{\#}\mu)&=\int_{\HH_n}d_H^p(q,\widehat{T}_{(\psi^{-1}(u)\ast\psi(v))}q)~d{\widehat{T}_{\psi(u)}}{}_{\#}\mu(q)\\
&=\int_{\HH_n}d_H^p(\widehat{T}_{\psi(u)}q',\widehat{T}_{(\psi^{-1}(u)\ast\psi(v))}\widehat{T}_{\psi(u)}q')~d\mu(q')\\
&=\int_{\HH^n}d_H^p(q'\ast\psi(u),q'\ast\psi(v)~\mathrm{d}\mu(q')\\
&=\int_{\HH^n}d_H^p(\psi(u),\psi(v))~\mathrm{d}\mu(q')\\
&=d_H^p(\psi(u),\psi(v)).
\end{split}
\end{equation}
From here we obtain $d_H(\psi(u),\psi(v))=\|u-v\|$ for all $u,v\in\mathbb{R}^{n+1}$, and thus the map $\psi$ is an isometric embedding of $\mathbb{R}^{n+1}$ into $\HH^n$, which contradicts the fact that $\HH^n$ is purely $(n+1)$-unrectifiable, \textcolor{black}{see \cite{Magnani} or }\cite{AK} for $n=1$ and \cite[Theorem 3]{BF} or \cite[Theorem 1.1]{BHW}. 

Since $\mathbb{R}^k$ embeds into $\mathcal{W}_p(\mathbb{R}^k)$ and $\mathcal{W}_p(\HH^k)$ 
\textcolor{black}{ 
for any $k \in \mathbb{N},$ and for $k\leq n$ the spaces $\mathcal{W}_p(\mathbb{R}^k)$ and $\mathcal{W}_p(\HH^k)$ embed isometrically into $\mathcal{W}_p(\HH^n)$ by the push-forward operations induced by the maps
$$
\R^k \ni (x_1,\dots,x_k) \mapsto (x_1, \dots,x_k, 0, \dots, 0; 0, \dots, 0; 0) \in \HH^n \simeq \R^{2n+1}
$$
and 
$$
\HH^k \ni (x_1,\dots,x_k;y_1,\dots, y_k;z) \mapsto (x_1, \dots,x_k, 0, \dots, 0; y_1, \dots,y_k, 0, \dots, 0; z) \in \HH^n
$$
respectively,} the following is an immediate consequence of \textcolor{black}{Theorem} \ref{embprop}.
\begin{corollary} Either of the spaces $\mathcal{W}_p(\mathbb{R}^k)$, or $\mathcal{W}_p(\HH^k)$ can be embedded isometrically into $\Wphn$ if and only if $k\leq n$.
\end{corollary}

\section{Isometric rigidity of $\Wphn$: Proof of Theorem 1.2.
}


Metric projections to vertical lines will play a key role in the proof.  Let us consider the vertical line $L_{\xtyt}=\{(\xt,\yt,z)\,|\,z\in\mathbb{R}\}$ and the metric projection $p_{\xtyt}:\HH^n\to L_{(\widetilde{x},\widetilde{y})}$. Let us calculate first the coordinates of $p_{\xtyt}(q_0)$, where $q_0=(x_0,y_0,z_0)\in\HH^n$. By definition,
\begin{equation*}
\begin{split}
    d_H^p(q_0,p_{\xtyt}(q_0))&=\min_{q'\in L_{(\widetilde{x},\widetilde{y})}}d_H^p(q_0,q')\\
    &=\min_{r\in\mathbb{R}}d_H^p\big((x_0,y_0,z_0),(\xt,\yt,r)\big)\\
    &=\textcolor{black}{\min_{r\in\mathbb{R}}\big\{\big||x_0-\xt|^2+|y_0-\yt|^2\big|^2+\big|z_0-r+2(y_0\cdot\xt-x_0\cdot\yt)\big|^2\big\}^{p/4}.}
\end{split}
\end{equation*}
This expression attains its minimum if and only if $r=z_0+2(y_0\cdot\xt-x_0\cdot\yt)$, so
\begin{equation}\label{projkoord}
p_{\xtyt}(x_0,y_0,z_0)=\big(\xt,\yt,z_0+2(y_0\cdot\xt-x_0\cdot\yt)\big).
\end{equation}
Our aim is to identify finitely supported measures by means of their projections to vertical lines. Inspired by \cite{BK}  we consider a version of the  \emph{Radon transformation} that we call {\it vertical Radon transform} as follows. For a measure $\mu\in\Wphn$ we call the mapping 
\begin{equation}\label{RadonTransform}
\mathcal{R}_{\mu}:\mathcal{L}\to\Wphn;\qquad\mathcal{R}_{\mu}(L_{(\widetilde{x},\widetilde{y})})=(p_{\xtyt})_{\#}\mu   
\end{equation}
the vertical Radon transform of $\mu$.
Let us mention that a similar object, called {\it Heisenberg X-ray transform} was studied in a recent paper by Flynn \cite{Flynn}. The main difference between our version of the Radon transform and the one considered in \cite{Flynn} is that in our case projections onto {\it vertical lines} define the {\it transform of a measure} while in \cite{Flynn} integrals on {\it horizontal lines} are used to define the {\it transform of a function}.

The following lemma says that the vertical Radon transform is injective on the set of finitely supported measures. Before stating this result we should mention that a similar statement in the Euclidean space is clearly false. Indeed, considering the Radon transform associated with the family of lines parallel to a fixed direction in the Euclidean space will never be injective. The fact, that this statement holds true in the Heisenberg setting is due to the "twirling effect" or the horizontal bundle of the Heisenberg group. In our case, this means that metric projections are rapidly turning by changing the location of the vertical line. 

\begin{lemma}\label{injectivity}
    Assume that $\nu=\sum_{i=1}^N m_i\delta_{q_i}\in\Wphn$ is an arbitrary finitely supported measure. If $\mathcal{R}_{\mu}=\mathcal{R}_{\nu}$ for some $\mu\in\Wphn$ then $\mu=\nu$.
\end{lemma}
\begin{proof}
First we show that for any finite set $\{q_1,\dots q_N\}\subset\HH^n$ there exists a vertical line $L_{(\widetilde{x},\widetilde{y})}$ such that the points $p_{\xtyt}(q_1),\dots,p_{\xtyt}(q_N)$ are pairwise different.  
Let us understand first what the equality $p_{\xtyt}(q)=p_{\xtyt}(q')$
 for two fixed distinct points $q=(x,y,z)$ and $q'=(x',y',z')$ tells us about $\xtyt$.
From \eqref{projkoord} we know that $p_{\xtyt}(q)=p_{\xtyt}(q')$ if and only if
\begin{equation}\label{projeq}
    z+2(y\cdot\xt-x\cdot\yt)=z'+2(y'\cdot\xt-x'\cdot\yt).
\end{equation}
 Since $x,y,x',y'$ are all fixed elements of $\mathbb{R}^n$ and $z-z'$ is a fixed constant, \eqref{projeq} is equivalent to the fact that $\xtyt\in\mathbb{R}^{2n}$ belongs to the affine hyperplane in $\mathbb{R}^{2n}$ described by
 \begin{equation}\label{hyperplane}
 (y-y')\cdot\xt+(x'-x)\cdot\yt=\frac{z'-z}{2}.
 \end{equation}
Let us denote this affine hyperplane by $H_{q,q'}$. We see that the requirement 
$$p_{(\xt, \yt)}(q) \neq p_{(\xt, \yt)}(q') $$ 
for a pair of distinct points $q, q'$, excludes exactly one affine hyperplane $H_{q,q'}$ in $\mathbb{R}^{2n}$.  Since we only have finitely many points, we get that 
$$\bigcup_{1\leq i<j\leq N}H_{q_i,q_j}\neq\mathbb{R}^{2n}.$$
But this means, that for any $\xtyt\notin\bigcup_{1\leq i<j\leq N}H_{q_i,q_j}$ the points $p_{\xtyt}(q_1),\dots,p_{\xtyt}(q_N)$ are pairwise different.

Next we show that $\mathcal{R}_{\mu}=\mathcal{R}_{\nu}$ implies $\supp(\mu)\subseteq\{q_1,\dots,q_N\}$. Let $q\in\HH^n$ be a point such that $q\notin\{q_1,\dots,q_N\}$. According to the argument above, we can find a vertical line $L_{(\widetilde{x},\widetilde{y})}$ such that $$q':=p_{\xtyt}(q),~q'_1:=p_{\xtyt}(q_1),~\dots,~q_N':=p_{\xtyt}(q_N)$$ are pairwise different. Since $(p_{\xtyt})_{\#}\nu$ is a finitely supported measure with $$\supp\big((p_{\xtyt})_{\#}\nu\big)=\{q_1',\dots,q_N'\}$$ and $\supp((p_{\xtyt})_{\#}\nu)=\supp((p_{\xtyt})_{\#}\mu)$, we get that $q'\notin\supp((p_{\xtyt})_{\#}\mu)$. In fact, a whole open neighborhood of $q'$ is disjoint from $\supp((p_{\xtyt})_{\#}\mu)$. Since the projection is continuous, we obtain that the inverse image of this open neighborhood is an open neighborhood of $q$, disjoint from $\supp(\mu)$. Since $q\notin\{q_1,\dots,q_N\}$ was arbitrary, we conclude that $\mu$ is finitely supported and $\supp(\mu)\subseteq\supp(\nu)=\{q_1,\dots,q_N\}$.
    
    Next, we show that $\mu=\nu$. Using the same vertical line $\xtyt\in\mathbb{R}^{2n}$ as before, we know that the projections $q_1',\dots, q_N'$ are pairwise different, and that $(p_{\xtyt})_{\#}\mu=(p_{\xtyt})_{\#}\nu$. This implies
$$(p_{\xtyt})_{\#}\mu(q_i')=(p_{\xtyt})_{\#}\nu(q_i')=\nu(q_i)=m_i$$
for all $1\leq i\leq N$. It follows from the choice of $L_{(\widetilde{x},\widetilde{y})}$ and the fact $\supp{\mu}\subseteq\{q_1,\dots,q_N\}$ that the inverse image of $q_i'$ under $(p_{\xtyt})_{\#}$ intersects $\supp(\mu)$ only in $q_i$, and therefore $\mu(q_i)=(p_{\xtyt})_{\#}\mu(q_i)=\nu(q_i)=m_i$ for all $1\leq i\leq N$. This implies that $\mu=\nu$.
\end{proof}

\color{black}
\begin{remark} \label{rem:radon-transform}
Note that if $\mathcal{R}_{\mu}=\mathcal{R}_{\nu}$ for some $\mu, \nu \in \Wphn,$ then by the definition \eqref{RadonTransform} we have $\ler{p_{\xtyt}}_{\#}\mu=\ler{p_{\xtyt}}_{\#}\nu$ for every $\xtyt \in \R^{2n}.$ This implies by the definition of the push-forward operation that $\mu\ler{p_{\xtyt}^{-1}\ler{\lers{\ler{\widetilde{x},\widetilde{y},\widetilde{z}}}}}=\nu\ler{p_{\xtyt}^{-1}\ler{\lers{\ler{\widetilde{x},\widetilde{y},\widetilde{z}}}}}$ for any $\ler{\widetilde{x},\widetilde{y},\widetilde{z}} \in \HH^n.$ Moreover, $p_{\xtyt}^{-1}\ler{\lers{\ler{\widetilde{x},\widetilde{y},\widetilde{z}}}}$ is the Heisenberg left translate of the horizontal hyperplane $\lers{\ler{x,y,z}\,\middle|\, z=0}$ to the point $\ler{\widetilde{x},\widetilde{y},\widetilde{z}}$ --- in other words, it is the horizontal hyperplane with characteristic point $\ler{\widetilde{x},\widetilde{y},\widetilde{z}}.$ Furtheremore, every $2n$-dimensional hyperplane in $\R^{2n+1}$ is a left translate of $\lers{\ler{x,y,z}\,\middle|\, z=0}$ by some vector except for the vertical hyperplanes. Therefore, the assumption $\mathcal{R}_{\mu}=\mathcal{R}_{\nu}$ implies that $\mu$ and $\nu$ agree on all hyperplanes in the Grassmanian $G(2n+1,2n)$ except for the lower dimensional family of vertical hyperplanes. 
\end{remark}
\color{black}

We continue with a brief remark concerning isometries of $p$-Wasserstein spaces over the real line. This remark will play a key role in the sequel.
\begin{remark}\label{real-line}
The Wasserstein space $\mathcal{W}_p(\mathbb{R})$ is isometrically rigid if and only if $p\neq 2$. Indeed, if $0<p<1$, then rigidity follows from \cite[Theorem 4.6]{GTV2}, as the metric $\varrho(x,y):=|x-y|^p$ satisfies the strict triangle inequality. In fact, for any complete and separable metric space $(X,\varrho)$ the $p$-Wasserstein space $\mathcal{W}_p(X)$ is isometrically rigid if $0<p<1$ (see \cite[Corollary 4.7]{GTV2}). As for the $p\geq1$ case: isometric rigidity of $\mathcal{W}_p(\mathbb{R})$ for $p=1$ has been proved in \cite[Theorem 3.7]{GTV1}, for the $p>1,~p\neq 2$ case see \cite[Theorem 3.16]{GTV1}. Finally, Kloeckner showed in \cite[Theorem 1.1.]{K} that $\mathcal{W}_2(\mathbb{R})$ is not isometrically rigid.
\end{remark}

Now we are ready to prove Theorem \ref{isom-rigidity}. Recall that this theorem says that $\Wphn$ is isometrically rigid for all $p>1$ and all $n\in\mathbb{N}$. That is, for any isometry $\Phi:\Wphn\to\Wphn$ there exists an isometry $\psi:\HH^n\to\HH^n$ such that $\Phi=\psi_{\#}.$\\

The strategy of the proof is the following. The first step is to show that isometries preserve the class of vertically supported measures. Then in the second step, we are going to show that measures supported on the same vertical line are mapped to measures that are also supported on the same vertical line. The third step is to show that $\Phi$ maps Dirac masses \textcolor{black}{to} Dirac masses, and thus we can assume without loss of generality that $\Phi(\delta_q)=\delta_q$ for all $q\in\HH^n$. Our aim from that point will be to show that $\Phi$ is the identity. In Step 4 we are going to prove that if  $\Phi$ fixes all Dirac measures then $\Phi$ fixes all vertically supported measures as well. This is the most involved part of the proof as we need to discuss the cases $p\neq 4$ and $p=4$ separately. The case $p=4$ is the difficult one, as in this case, we need an additional argument to rule out the existence of exotic isometries and a nontrivial shape-preserving isometry that maps every measure to its symmetric with respect to its center of mass. Finally, in the last step we are going to use the Radon transform to show that $\Phi$ fixes all finitely supported measures, and thus all elements of $\Wphn$. We remark that this scheme cannot be applied in the $p=1$ case, since it heavily relies on the explicit description of complete geodesics (see Proposition \ref{newprop1}), and such a nice description is not available in the $p=1$ case according to Remark \ref{p=1 geod}.\\

\noindent\underline{Step 1.} 
Let $\mu \in \mathcal{W}_p(\mathbb{H}^n)$ be a vertically supported measure. We are going to show that $\Phi(\mu)$ is also vertically supported. Assume by contradiction that the measure $\nu= \Phi(\mu)$ is not vertically supported. Consider the geodesic ray 
$$\gamma:[0,\infty)\to\Wphn;\qquad\gamma(t)=(\widehat{D}_t)_{\#}\nu$$
that contains $\nu$ but $\nu$ is not its endpoint. Define $\widetilde{\gamma}: [0, \infty) \to \mathcal{W}_p(\mathbb{H}^n)$ as  $\widetilde{\gamma}(t):= \Phi^{-1}\gamma(t)$. This is a geodesic ray containing $\mu= \widetilde{\gamma}(1)$. Since $\mu$ is vertically supported, and it is not the endpoint of the geodesic ray $\widetilde{\gamma}$, applying Proposition \ref{vertically-supported} it follows that $\widetilde{\gamma}$ can be extended to a complete geodesic ray. Taking the image of this extension by $\Phi$ we conclude that also $\gamma$ itself can be extended to a complete geodesic ray. On the other hand, Proposition \ref{newprop1} implies that if a measure on a complete geodesic ray is vertically supported then all measures on the complete geodesic must have this property. This gives a contradiction since $\gamma(0)$ is vertically supported and $\nu= \gamma(1)$ is not.\\

\noindent\underline{Step 2.} Now suppose that $\mu_1, \mu_2 \in \mathcal{W}_p(\mathbb{H}^n)$ are two measures supported on the same vertical line.  Then their image measures $\Phi(\mu_1)$ and $\Phi(\mu_2)$ are also supported on the same vertical line. Indeed, by composing $\Phi$ by isometries of $\mathcal{W}_p(\mathbb{H}^n)$ induced by left translations of $\mathbb{H}^n$ we can assume that $\mu_1, \mu_2$ and also $\Phi(\mu_1)$ are supported on the $0z$-axis. Assume by contradiction that $\Phi(\mu_2)$ is supported on the vertical line $L_{(u,v)}$, where $U:=(u,v,0)\in\HH^n$ is a non-zero horizontal vector. Consider the complete geodesic 
$$\gamma:\mathbb{R} \to \mathcal{W}_p(\mathbb{H}^n);\qquad\gamma(t) :=(\widehat{T}_{tU})_{\#}\textcolor{black}{\Phi}(\mu_2).$$
Note that since $\Phi(\mu_2)$ is supported on the vertical line $L_{(u,v)}$, therefore $\gamma(-1)=(\widehat{T}_{-U})_{\#} \textcolor{black}{\Phi}(\mu_2)$ is supported on the $0z$ axis. Denoting by $\nu= \gamma(-1)$ we can also write $\widetilde{\gamma}(t) = (\widehat{T}_{tU})_{\#} \nu$ as a reparametrisation of $\gamma$.
 
For two non-empty subsets $A, B \subset \mathcal{W}_p(\mathbb{H}^n)$, the distance between $A$ and $B$ in $\Wphn$ is defined as  

$$\mbox{dist}_{d_{\mathcal{W}_p}}(A,B) := \inf \{ d_{\mathcal{W}_p}(\mu,\nu): \mu \in A, \nu \in B \}.$$
	
Let us observe that if  $\Phi: \mathcal{W}_p(\mathbb{H}^n) \to \mathcal{W}_p(\mathbb{H}^n)$ is  an isometry, then we have the equality 
	$$ \mbox{dist}_{d_{\mathcal{W}_p}}(A,B) = \mbox{dist}_{d_{\mathcal{W}_p}}(\Phi(A),\Phi(B)), \ \text{for all} \ A, B \subseteq \mathcal{W}_p(\mathbb{H}^n) .$$

In the following, we claim that if  $\mu_1, \mu_2 \in \mathcal{W}_p(\mathbb{H}^n)$ are two measures supported on the vertical axis and $U=(u,v,0)$ is a horizontal vector, then the following inequality holds:
\begin{equation} \label{vertical-ineq}
d_{\mathcal{W}_p}(\mu_1, \mu_2) \leq d_{\mathcal{W}_p}(\mu_1, (\widehat{T}_U)_{\#}\mu_2) ,
\end{equation}
with equality \textcolor{black}{holding} if and only if $U = (0,0,0)$. 

 To prove the above claim, let us observe first that for any $q= (0,0,z)$ and $q'=(u, v, z') $ we have the inequality 
$$ d^p_H(q, \widehat{T}_{(-U)}q') = |z-z'| ^{\frac{p}{2}}\leq \{\textcolor{black}{(|u|^2+|v|^2)^2} + (z-z')^2 \}^{\frac{p}{4}} = d^p_H(q, q').$$

Let $\widetilde{\Pi}$ be an arbitrary coupling of $\mu_1$ and  $(\widehat{T}_U)_{\#}\mu_2$. Define $\widehat{\Pi}:= (\widehat{\mathrm{Id}}\times \widehat{T}_{(-U)})_{\#} \widetilde{\Pi}$. It is easy to see that $\widehat{\Pi}$ is a coupling of $\mu_1$ and $\mu_2$. Then we have:

\begin{equation}
\begin{split}
d^p_{\mathcal{W}_p}(\mu_1, \mu_2) & \leq \int_{\mathbb{H}^n\times \mathbb{H}^n}d_H^p(q,r) ~\mathrm{d}\widehat{\Pi}(q,r)\\
&=\int_{\mathbb{H}^n\times \mathbb{H}^n}d_H^p(q,\widehat{T}_{(-U)} q') ~\mathrm{d}\widetilde{\Pi}(q,q')\\
  &\leq \int_{\mathbb{H}^n\times \mathbb{H}^n}d_H^p(q, q') ~\mathrm{d}\widetilde{\Pi}(q,q').
\end{split}
\end{equation}
Taking the infimum over all couplings $\widetilde{\Pi}$ of $\mu_1$ and $(T_U)_{\#}\mu_2$ we conclude the proof of the claim.

\textcolor{black}{On the one hand,} according to \eqref{vertical-ineq}, \textcolor{black}{and since $U=(u,v,0)\neq(0,0,0)$ by assumption} we can write 
\begin{equation} \label{geod-distance}
D:= \mbox{dist}_{d_{\mathcal{W}_p}}( \gamma, \Phi(\mu_1)) = \mbox{dist}_{d_{\mathcal{W}_p}}( \widetilde{\gamma}, \Phi(\mu_1)) = 
d_{\mathcal{W}_p}( \nu, \Phi(\mu_1)) < d_{\mathcal{W}_p}( \Phi(\mu_2), \Phi(\mu_1)) .
\end{equation}

On the other hand, since $\Phi^{-1}$ is also an isometry, we have 
\begin{equation} \label{geod-distance-inverse}
D= \mbox{dist}_{d_{\mathcal{W}_p}}( \Phi^{-1}(\gamma), \mu_1) .
\end{equation}

Since $\Phi^{-1}(\gamma)$ is a complete geodesic through $\mu_2$ and since both $\mu_1$ and $\mu_2$ lie on the $0z$- axis we can apply again \eqref{vertical-ineq} to obtain
\begin{equation} \label{geod-distance-inverse2}
D= \mbox{dist}_{d_{\mathcal{W}_p}}( \Phi^{-1}(\gamma), \mu_1) = d_{\mathcal{W}_p}( \mu_2, \mu_1) .
\end{equation}
	It is clear that relations \eqref{geod-distance-inverse} and \eqref{geod-distance-inverse2} are in contradiction, finishing the proof of the second step.\\
 
\noindent\underline{Step 3.} From the first step it follows that Dirac masses are mapped to vertically supported measures. Let us consider all measures supported on a given fixed vertical line. According to the previous step their images lie on the same vertical line. Without loss of generality, we can assume that both vertical lines are the $0z$-axis. On the other hand, the Heisenberg metric squared restricted to the $0z$-axis coincides with the usual metric on the real line $\mathbb{R}$, which implies that if we restrict the isometry $\Phi$ on the measures with support on the $0z$-axis then it will coincide with an isometry $F$  of $\mathcal{W}_{\frac{p}{2}}(\mathbb{R})$. According to Remark \ref{real-line} the Wasserstein space $\mathcal{W}_{\frac{p}{2}}(\mathbb{R})$ is isometrically rigid except for the case $p=4$. The $p=4$ corresponds to the $\mathcal{W}_2(\mathbb{R})$ case, where we do not have rigidity, but Dirac masses are preserved by isometries even in that case. This means in particular that $F$ maps Dirac masses to Dirac masses and so does $\Phi$ as well. 

We can conclude that $\Phi(\delta_q)= \delta_{\psi(q)}$ for some mapping $\psi: \mathbb{H}^n \to \mathbb{H}^n$.
	Furthermore, if we pick two points $q_1, q_2 \in \mathbb{H}^n$ then we have 
 \begin{equation*}
 \begin{split}
d_H(q_1, q_2) &=  d_{\mathcal{W}_p}( \delta_{q_1}, \delta_{q_2})\\
&= d_{\mathcal{W}_p}( \Phi(\delta_{q_1}), \Phi(\delta_{q_2}))\\
&=d_{\mathcal{W}_p}( \delta_{\psi(q_1)}, \delta_{\psi(q_2)})\\
&= d_H(\psi(q_1), \psi(q_2)),
\end{split}
 \end{equation*}
which shows that $\psi: \mathbb{H}^n \to \mathbb{H}^n$ is an isometry. Since $\Phi=\psi_{\#}$ if and only if $\Phi\circ\psi_{\#}^{-1}= \mathrm{Id}$, we can assume without loss of generality that $\Phi(\delta_q)=\delta_q$ for all $q\in\HH^n$. Our aim now is to show that $\Phi(\mu)=\mu$ for all $\mu\in\Wphn$.\\

\noindent\underline{Step 4.} Recall that if $L$ is a vertical line then $\mathcal{W}_p(L)$ denotes the set of those $\mu\in\Wphn$ such that $\supp\mu\subseteq L$. We know that measures supported on the same vertical line are mapped into measures supported on the same (possibly different) line. In our case, since Dirac measures are fixed, we get that $\Phi|_{\Wpl}:\Wpl\to\Wpl$ is a bijection, in fact, it coincides with an isometry of $\mathcal{W}_{\frac{p}{2}}(\mathbb{R})$. Since $\mathcal{W}_{\frac{p}{2}}(\mathbb{R})$ is isometrically rigid if $p \neq 4$, since Dirac measures are fixed, we conclude that all measures supported on $L$ will be fixed. Since $L$ is arbitrary we can conclude that all vertically supported measures will be fixed by $\Phi$ if $p\neq4$.

If $p=4$ then $\Phi(\delta_q)=\delta_q$ for all $q\in L$ itself does not imply $\Phi(\mu)=\mu$ for all $\mu\in\mathcal{W}_4(L)$, as $\mathcal{W}_4(L)$ corresponds to $\mathcal{W}_{\frac{4}{2}}(\mathbb{R})$ and Kloeckner showed that $\mathcal{W}_2(\mathbb{R})$ admits \textcolor{black}{non-trivial isometries that fix all Dirac measures and are different from the identity}. Our purpose is to rule out these isometries. This boils down to investigating the action of the isometry on vertically supported measures whose support consists of two points. To keep the presentation precise, let us introduce some notations. 
For a fixed vertical line $L_{(x,y)}$ the set of measures supported on two points of $L_{(x,y)}$ will be denoted by $\Delta_2^{(x,y)}$
\begin{equation}\label{delta2xy}
    \Delta_2^{(x,y)}=\{\alpha\delta_q+(1-\alpha)\delta_{q'}\,|\, \alpha\in(0,1),\, q,q'\in L_{(x,y)}\}.
\end{equation}
Following the notations in  Kloeckner's paper \cite{K}, elements of $\Delta_2^{(x,y)}$ will be parametrized by three parameters $m\in\mathbb{R}$, $\sigma\geq0$, and $r\in\mathbb{R}$ as follows:
\begin{equation}\label{parametrization}
\mu_{(x,y)}(m,\sigma,r)=\frac{e^{-r}}{e^r+e^{-r}}\delta_{(x,y,m-\sigma e^r)}+\frac{e^{r}}{e^r+e^{-r}}\delta_{(x,y,m+\sigma e^{-r})}.
\end{equation}
\textcolor{black}{A very important property of $\Delta_2^{(x,y)}$ is that its geodesic convex hull is dense in $\mathcal{W}_4(L_{(x,y)})$, which is again, can be identified with $\mathcal{W}_2(\mathbb{R})$. Therefore, if an isometry is given, it is enough to know its action on $\Delta_2^{(x,y)}$. According to Lemma 5.2 in \cite{K}, if an isometry $\Phi$ acting on $\mathcal{W}_4(\HH^n)$ fixes all Dirac masses, then for every $(x,y)\in \R^{2n},$ its restriction $\Phi_{(x,y)}:=\Phi|_{\mathcal{W}_4(L_{(x,y)})}$ admits the following form:
$$
\Phi_{(x,y)}:\mathcal{W}_4(L_{(x,y)})\to\mathcal{W}_4(L_{(x,y)}),\quad\Phi_{(x,y)}\big(\mu_{(x,y)}(m,\sigma,r)\big):=\mu_{(x,y)}(m,\sigma,\varphi_{(x,y)}(r))$$
where $\varphi_{(x,y)}:\R\to\R$ is an isometry. This means that $\Phi_{(x,y)}$ can be equal to the exotic isometry}
\begin{equation}\label{eq:Phit}
\textcolor{black}{\Phi_{(x,y)}^{(t)}:\mathcal{W}_4(L_{(x,y)})\to\mathcal{W}_4(L_{(x,y)}),\quad \textcolor{black}{\Phi_{(x,y)}^{(t)}} \big(\mu_{(x,y)}(m,\sigma,r)\big):=\mu_{(x,y)}(m,\sigma,r+t)}
\end{equation}
\textcolor{black}{for some $t\neq0$, that is, $\varphi_{(x,y)}$ is the translation of the real line by $t;$ or $\Phi_{(x,y)}$ is equal to the shape-preserving isometry}
\begin{equation}\label{eq:Phistar}
\textcolor{black}{\Phi_{(x,y)}^*:\mathcal{W}_4(L_{(x,y)})\to\mathcal{W}_4(L_{(x,y)}),\quad \Phi_{(x,y)}^* \big(\mu_{(x,y)}(m,\sigma,r)\big):=\mu_{(x,y)}(m,\sigma,-r),}
\end{equation} 
\textcolor{black}{that is, $\varphi_{(x,y)}$ is the reflection of $\R$ with center $0$; or $\Phi_{(x,y)}$ is the composition of the above two, $\Phi_{(x,y)}^{(t)}\circ\Phi_{(x,y)}^*$ for some $t\neq0$. A priory, it could happen that for different $(x,y)$ the action of $\Phi_{(x,y)}$ are different. The next argument will show that this is not the case, and these actions are uniform in the sense that $\Phi_{(x,y)}$ has the same form as $\Phi_{(0,0)}$ for all $(x,y)$.}

\textcolor{black}{Observe that any isometry $\Phi$ commutes with the push-forward induced by the metric projection $p_{(x,y)_{\#}}$ for all $(x,y)\in\mathbb{R}^{2n}$:
\begin{equation}\label{commutation}
    {p_{(x,y)}}_{\#}\big(\Phi(\mu)\big)=\Phi\big({p_{(x,y)}}_{\#}\mu\big)\quad\mbox{\textcolor{black}{for all $\mu\in\mathcal{W}_4(\HH)$}}.
\end{equation}
Indeed, by definition, 
\begin{equation*}
{p_{(x,y)}}_{\#}\big(\Phi(\mu)\big)=\arg\min_{\eta}\Big\{\dwp(\eta,\Phi(\mu))\,\Big|\,\eta\in\mathcal{W}_{\textcolor{black}{4}}(L_{(x,y)})\Big\}.
\end{equation*}
Assume by contradiction that there exists an $\eta\in\mathcal{W}_{\textcolor{black}{4}}(L_{(x,y)})$ such that 
$$d_{\mathcal{W}_{\textcolor{black}{4}}}(\eta,\Phi(\mu))<d_{\mathcal{W}_{\textcolor{black}{4}}}(\Phi\big({p_{(x,y)}}_{\#}\mu),\Phi(\mu)).$$
Since $\Phi|_{\mathcal{W}_{\textcolor{black}{4}}(L_{(x,y)})}$ is bijective, there exists an $\eta'\in\mathcal{W}_{\textcolor{black}{4}}(L_{(x,y)})$ such that $\Phi(\eta')=\eta$, and thus,
$$d_{\mathcal{W}_{\textcolor{black}{4}}}(\Phi(\eta'),\Phi(\mu))<d_{\mathcal{W}_{\textcolor{black}{4}}}(\Phi\big({p_{(x,y)}}_{\#}\mu),\Phi(\mu)).$$
Since $\Phi$ is an isometry, this is equivalent to
$$d_{\mathcal{W}_{\textcolor{black}{4}}}(\eta',\mu)<d_{\mathcal{W}_{{4}}}({p_{(x,y)}}_{\#}\mu,\mu),$$
a contradiction. This guarantees that $\Phi$ acts uniformly on all vertical lines. Now, we turn to prove that the restriction of $\Phi$ to any vertical lines cannot be one of the above non-trivial isometries.}

\textcolor{black}{First assume that $\Phi_{(x,y)}=\Phi|_{\mathcal{W}_4(L_{(x,y)})}$ is equal to $\Phi_{(x,y)}^{(t)}$ for some $t\in\mathbb{R}$. Our goal is to show that $t=0$ by showing that $\Phi$ cannot be isometric otherwise.}

Recall that the projection of $(x,y,z)\in\HH^n$ onto the vertical line $L_{(\widetilde{x},\widetilde{y})}$ is
$$p_{(\widetilde{x},\widetilde{y})}(x,y,z)=\big(\widetilde{x},\widetilde{y},z+2(y\cdot\widetilde{x}-x\cdot\widetilde{y})\big),$$
and the inverse image of a point \textcolor{black}{$(\widetilde{x},\widetilde{y},\widetilde{z})$ by the projection $p_{(\xt,\yt)}$ is the horizontal hyperplane with characteristic point $(\widetilde{x},\widetilde{y},\widetilde{z}),$ that is,}
\begin{equation}\label{inverse image}
    p_{(\widetilde{x},\widetilde{y})}^{-1}(\xt,\yt,\zt)=\big\{(x,y,z)\,|\,\zt=z+2(y\cdot\xt-x\cdot\yt)\big\}.
\end{equation}
Let us consider the measure $\mu=\frac{1}{2}\big(\delta_{(u,0,0)}+\delta_{(0,u,0)}\big)$, where $u=(1,0,\dots,0)\in\mathbb{R}^n$ and $0=(0,\dots,0)\in\mathbb{R}^n$. 
First we show that $\supp\big(\Phi(\mu)\big)\subseteq H_{[z=0]}=\{(x,y,0)\in\HH^n\,|\,(x,y)\in\mathbb{R}^{2n}\}$. Since ${p_{(0,0)}}_{\#}\mu=\delta_{(0,0,0)}$ and \textcolor{black}{$\Phi_{(0,0)}^{(t)}$} fixes Dirac masses, it follows from \eqref{commutation} that \begin{equation}
\delta_{(0,0,0)}=\textcolor{black}{\Phi_{(0,0)}^{(t)}}(\delta_{(0,0,0)})={\Phi(p_{(0,0)}}_{\#}\mu)={p_{(0,0)}}_{\#}(\Phi(\mu))
\end{equation}
and thus, $\supp\big(\Phi(\mu)\big)\subseteq p_{(0,0)}^{-1}\{(0,0,0)\}=H_{[z=0]}.$

Now let us consider the projection onto $L_{(u,u)}$. Since 
$p_{(u,u)}(u,0,0)=(u,u,-2)$ and
$p_{(u,u)}(0,u,0)=(u,u,2)$,
we get that ${p_{(u,u)}}_{\#}\mu=\mu_{(u,u)}(0,2,0)$, and thus $\Phi({p_{(u,u)}}_{\#}\mu)=\mu_{(u,u)}(0,2,t)$. Again, using that $\Phi$ commutes with projections we get:
\begin{equation}\label{eq:invproj1}
\supp\big(\Phi(\mu)\big)\subseteq\left(p_{(u,u)}^{-1}(u,u,-2e^t)\cup{p_{(u,u)}^{-1}(u,u,2e^{-t})}\right).
\end{equation}
Combining this with $\supp\big(\Phi(\mu)\big)\subseteq H_{[z=0]}$ we obtain that
\begin{equation}\label{proj1}
    \supp\big(\Phi(\mu)\big)\subseteq\{(x,y,0)\,|\,y_1=x_1-e^{t}\}\cup\{(x,y,0)\,|\,y_1=x_1+e^{-t}\}=S_+
\end{equation}
where $x_1$ and $y_1$ are the first coordinates of $x\in\mathbb{R}^n$ and $y\in\mathbb{R}^n$, respectively.

The same calculation with the projection $p_{(-u,-u)}$ leads to $$\Phi({p_{(-u,-u)}}_{\#}\mu)=\mu_{(-u,-u)}(0,2,t),$$ and thus
\begin{equation}\label{eq:invproj2}
\supp\big(\Phi(\mu)\big)\subseteq\left(p_{(-u,-u)}^{-1}(-u,-u,-2e^t)\cup{p_{(-u,-u)}^{-1}(-u,-u,2e^{-t})}\right).
\end{equation}

The intersection of this set with $H_{[z=0]}$ is again the union of two affine hyperplanes in $\mathbb{R}^{2n}$, which leads to:
\begin{equation}\label{proj2}
    \supp\big(\Phi(\mu)\big)\subseteq\{(x,y,0)\,|\,y_1=x_1+e^{t}\}\cup\{(x,y,0)\,|\,y_1=x_1-e^{-t}\}=S_-.
\end{equation}
We see that $\supp\big(\Phi(\mu)\big)\subseteq S_+$ and  $\supp\big(\Phi(\mu)\big)\subseteq S_-$, but $S_+$ and $S_-$ are disjoint unless $t=0$.
\par
Now we have to exclude the \textcolor{black}{case that $\Phi_{(x,y)}=\Phi|_{\mathcal{W}_4(L_{(x,y)})}=\Phi_{(x,y)}^*$ for all $(x,y) \in \R^{2n},$ that is, $\Phi$ acts by the shape-preserving action defined in \eqref{eq:Phistar} on every vertical line}. The technique we use to show that \textcolor{black}{such a $\Phi$} cannot be isometric is similar to the approach we used above to prove that \textcolor{black}{$\Phi$ is not isometric if $\Phi_{(x,y)}=\Phi|_{\mathcal{W}_4(L_{(x,y)})}=\Phi_{(x,y)}^{(t)}$ for all $(x,y) \in \R^{2n}$ for some $t\neq0$.}
Consider the measure
$$
\mu=\frac{1+\alpha}{2} \delta_{(u,0,0)}+\frac{1-\alpha}{2} \delta_{(0,u,0)} \qquad (\alpha \in (0,1)).
$$
The projection onto $L_{(u,0)}$ is
$$\ler{p_{(u,0)}}_{\#} \mu=\frac{1+\alpha}{2} \delta_{(u,0,0)}+\frac{1-\alpha}{2} \delta_{(u,0,2)}.$$
The mean of $\ler{p_{(u,0)}}_{\#} \mu$ is $1-\alpha$, and the reflection to $1-\alpha$ sends $0$ to $2(1-\alpha)$ and sends $2$ to $-2\alpha$, thus
$$
\ler{\textcolor{black}{\Phi_{(u,0)}^*} \circ \ler{p_{(u,0)}}_{\#}} \mu=\frac{1+\alpha}{2} \delta_{(u,0,2(1-\alpha))}+\frac{1-\alpha}{2} \delta_{(u,0,-2 \alpha)}.
$$
Very similarly,
$$
\ler{\textcolor{black}{\Phi_{(0,u)}^*} \circ \ler{p_{(0,u)}}_{\#}} \mu=\frac{1+\alpha}{2} \delta_{(0,u,-2\alpha))}+\frac{1-\alpha}{2} \delta_{(\textcolor{black}{0,u},-2 (1+\alpha)},
$$
and
$$
\ler{\textcolor{black}{\Phi_{(u,u)}^*} \circ \ler{p_{(u,u)}}_{\#}} \mu=\frac{1+\alpha}{2} \delta_{(u,u,-4\alpha+2))}+\frac{1-\alpha}{2} \delta_{(u,u,-4\alpha-2)}.
$$
Considering the pre-images of the projections $p_{(u,0)}, p_{(0,u)},$ and $p_{(u,u)}$ --- very similarly to the computations between \eqref{eq:invproj1} and \eqref{proj2} --- we get that the support of $\textcolor{black}{\Phi(\mu)}$ is a subset of $\{(\alpha u,(1-\alpha)u,0),((1+\alpha)u,-\alpha u,0)\}$, and taking the weights into consideration, there is no other possibility than
$$
\textcolor{black}{\Phi(\mu)}=\frac{1+\alpha}{2}\delta_{(\alpha u,(1-\alpha) u,0)}+\frac{1-\alpha}{2}\delta_{((1+\alpha) u,-\alpha u,0)}.
$$
However,
$$
d_{\cW_4}^4\ler{\textcolor{black}{\Phi(\mu)},\delta_{(0,0,0)}}=\frac{1+\alpha}{2}\ler{\alpha^2+(1-\alpha)^2}^2+\frac{1-\alpha}{2} \ler{(1+\alpha)^2+(-\alpha)^2}^2
$$
which takes, e.g., the value $7/4$ for $\alpha=\pm 1/2,$ while $d_{\cW_4}^4\ler{\mu,\delta_{(0,0,0)}}=1$ is clear for all values of $\alpha.$ We concluded that $\textcolor{black}{\Phi}$ is not an isometry. 
\par 
Finally, we have to exclude the third possibility, that is, when $\Phi_{(x,y)}=\Phi_{(x,y)}^{(t)}\circ\Phi_{(x,y)}^{*}$ for some $t\neq0$ for all $(x,y)$. Observe that the measures $\mu=\frac{1}{2}\big(\delta_{(u,0,0)}+\delta_{(0,u,0)}\big)$ and $\delta_{(0,0,0)}$ are fixed points of $\Phi^*$, and thus the argument above showing that $\Phi_{(x,y)}^{(t)}$ is not isometric for $t\neq0$ shows that $\Phi_{(x,y)}^{(t)}\circ \Phi^*$ is not isometric either. Now we know that an isometry $\Phi: \mathcal{W}_4(\HH^n) \to \mathcal{W}_4(\HH^n)$ fixing all Dirac masses acts identically on measures supported on vertical lines even in the $p=4$ case.


\noindent\underline{Step 5.} In order to prove isometric rigidity, it is enough to show that $\Phi(\mu)=\mu$ for all $\mu\in\mathcal{F}(\HH^n)$, \textcolor{black}{where $\mathcal{F}(\HH^n)$ denotes the set of finitely supported measures. Let $\mu\in\mathcal{F}(\HH^n)$} and $\xtyt\in\mathbb{R}^{2n}$ an arbitrary vector and the associated vertical line $L_{(\widetilde{x},\widetilde{y})}$.
Since $\Phi$ fixes vertically supported measures and commutes with the push-forward induced by projections onto vertical lines we obtain:
$$ (p_{(\xt,\yt)})_{\#}\Phi(\mu)= \Phi((p_{(\xt,\yt)})_{\#}\mu)=(p_{(\xt,\yt)})_{\#}\mu.$$

Since $L_{(\widetilde{x},\widetilde{y})}$ was arbitrary, this implies that $\mathcal{R}_{\mu}=\mathcal{R}_{\Phi(\mu)}$. Since $\mu$ is finitely supported, Lemma \ref{injectivity} implies that $\Phi(\mu)=\mu$.
\medskip

\paragraph*{{\bf Acknowledgements}} 
We thank the anonymous reviewers for their insightful suggestions and comments. In particular, Remark \ref{rem:radon-transform} is based completely on their observations.
This work was initiated at the thematic semester  ``Optimal transport on quantum structures'' (Fall 2022) at the Erd\H{o}s Center, R\'enyi Institute, Budapest. Zolt\'an M. Balogh would like to thank the R\'enyi Institute for the invitation to participate in this event and for the kind hospitality. 
We are grateful to Nicolas Juillet and Katrin F\"assler for their inputs and for several discussions on the topic.


\end{document}